\documentclass[a4paper,10pt,reqno]{amsart}

\usepackage[T1]{fontenc}
\usepackage{graphicx}

\usepackage[english]{babel}

\usepackage{soul}
\usepackage{xcolor}
\usepackage{subfiles}

\usepackage{graphicx}
\usepackage{amsmath,amsfonts,amsthm,amssymb}
\usepackage{wrapfig}
\usepackage[shortlabels]{enumitem}
\usepackage{tikz}
\usepackage{esint}
\usepackage{mathrsfs}
\usepackage{physics}
\usepackage[autostyle]{csquotes}
\usepackage{hyperref}
\usepackage{microtype}
\usepackage{stackengine,scalerel}
\usepackage{xcolor}
\hypersetup{
    colorlinks,
    linkcolor={red!50!black},
    citecolor={blue!50!black},
    urlcolor={blue!80!black}
}

\usepackage[normalem]{ulem} 
\usepackage{mwe}
\usepackage{bbm}
\usepackage{mathtools}
\usepackage{emptypage}

\newcommand{\R}{\mathbb{R}}
\newcommand{\N}{\mathbb{N}}

\newcommand{\ball}{\operatorname{B}}

\newcommand{\lebe}{\operatorname{L}}

\newcommand{\sobo}{\operatorname{W}}

\newcommand{\hold}{\operatorname{C}}
\newcommand{\vari}{\operatorname{BV}}
\newcommand{\D}{\mathrm{D}}
\newcommand{\B}{\operatorname{B}}
\newcommand{\eps}{\varepsilon}

\newcommand{\bv}{\operatorname{BV}}
\newcommand{\wstar}{\stackrel{*}{\rightharpoonup}}

\newcommand{\nocontentsline}[3]{}
\let\origcontentsline\addcontentsline
\newcommand\stoptoc{\let\addcontentsline\nocontentsline}
\newcommand\resumetoc{\let\addcontentsline\origcontentsline}

\allowdisplaybreaks

\newlength{\leftstackrelawd}
\newlength{\leftstackrelbwd}
\def\leftstackrel#1#2{\settowidth{\leftstackrelawd}%
{${{}^{#1}}$}\settowidth{\leftstackrelbwd}{$#2$}%
\addtolength{\leftstackrelawd}{-\leftstackrelbwd}%
\leavevmode%
{\kern-.5\leftstackrelawd}{}\mathrel{\mathop{#2}\limits^{#1}}}

\newtheorem{defi}{Definition}
\newtheorem{lemma}[defi]{Lemma}

\newtheorem{satz}[defi]{Theorem}
\newtheorem{beisp}[defi]{Example}
\newtheorem{remark}[defi]{Remark}
\newtheorem{prop}[defi]{Proposition}
\numberwithin{defi}{section}
\numberwithin{equation}{section}

\graphicspath{ {images/} }
\usepackage{caption}
\usepackage{subcaption}
\usepackage{float}
\usepackage[width=150mm,top=30mm,bottom=25mm,bindingoffset=6mm]{geometry}
\usepackage{fancyhdr}
\usepackage{setspace}
\usepackage{pgfplots}
\pgfplotsset{compat=1.18}
\usepackage{tikz}

\definecolor{seeblau}{RGB}{0, 169, 224}
\setulcolor{seeblau}
\setul{0.1ex}{0.2ex}

\def\XXint#1#2#3{{\setbox0=\hbox{$#1{#2#3}{\int}$}
     \vcenter{\hbox{$#2#3$}}\kern-.5\wd0}}

\title[On the singular set of $\bv$ minimizers]{On the singular set of $\bv$ minimizers \\  for non-autonomous functionals}
\author[L. Fu\ss angel]{Lukas Fu\ss angel}
\address{Fachbereich Mathematik und Statistik, Universit\"{a}t Konstanz, Universit\"{a}tsstraße 10, 78464 Konstanz, Germany}
\email{lukas.fussangel@uni-konstanz.de}
\author[B. Priyasad]{Buddhika Priyasad}
\email{priyasad@uni-konstanz.de}
\author[P. Stephan]{Paul Stephan}
\email{paul.stephan@uni-konstanz.de}

\date{September 11, 2025}
\keywords{Sobolev regularity, dimension reduction, linear growth functionals, functions of bounded variation}
\subjclass[2020]{35B65, 35J47, 49N60}

\begin{document}
\begin{abstract}
{
We investigate regularity properties of minimizers for non-autonomous convex variational integrands $F(x, \D u)$ with linear growth, defined on bounded Lipschitz domains $\Omega \subset \R^n$. Assuming appropriate ellipticity conditions and Hölder continuity of $\D_zF(x,z)$ with respect to the first variable, we establish higher integrability of the gradient of minimizers and provide bounds on the Hausdorff dimension of the singular set of minimizers.
}
\end{abstract}
\maketitle
\setcounter{tocdepth}{1}
\tableofcontents

\section{Introduction}
\subsection{Convex variational integrals on $\bv$.}
Let $\Omega \subset \R^n$ be an open and bounded Lipschitz domain. In this paper, we study the regularity properties of minimizers of non-autonomous convex variational integrals of the form
\begin{align}\label{variationalproblem}
\mathscr F[u] = \mathscr F[u; \Omega]:= \int \limits_{\Omega} F(x, \D u) \dd{x}, \quad u \in \sobo^{1,1}_{u_0}(\Omega; \R^N),
\end{align}
where $\sobo^{1,1}_{u_0}(\Omega; \R^N)= u_0 + \sobo^{1,1}_0(\Omega; \R^N)$ for some Dirichlet boundary datum $u_0 \in \sobo^{1,1}(\Omega; \R^N)$. Here we have $F \in \mathrm{C}( \bar{\Omega} \times \R^{N \times n})$ with $F(x, \cdot) \in \hold^2(\R^{N\times n})$ for every $x \in \Omega$ and $F$ is an integrand of linear growth. This means that there are constants $0<c_0\le c_1 <\infty$ and $c_2 \ge 0$ such that
\begin{align}\label{eq:lineargrowth}
c_0 |z| - c_2 \le F(x, z) \le c_1 (1+|z|)
\end{align}
for all $z \in \R^{N \times n}$ and all $x \in \Omega$. \\

In order to obtain the existence of minimizers, it is necessary to relax the above functional to the space of functions of bounded variation. This is necessary because bounded sequences in $\sobo^{1, 1}(\Omega; \mathbb{R}^N)$ typically do not have weakly converging subsequences; rather, the gradients of minimizing sequences may concentrate. More precisely, since any minimizing sequence $(u_j)_j$ for \eqref{variationalproblem} is bounded in $\sobo^{1,1}(\Omega; \R^N)$ due to the linear growth assumption \eqref{eq:lineargrowth}, there is a subsequence $(u_{j_k})_k$ and a function $u \in \bv(\Omega; \R^N)$ such that $u_{j_k} \to u$ in $\lebe^{1}(\Omega; \R^N)$ and $\D u_{j_k} \wstar \D u$. Here, $\bv(\Omega; \R^N)$ is the set of all $u \in \lebe^1(\Omega; \R^N)$ such that the distributional derivative $\D u$ can be represented by a finite $\R^{N \times n}$-valued Radon measure on $\Omega$. However, there need not exist a subsequence of $(u_j)_j$ converging (weakly) in $\sobo^{1,1}(\Omega;\R^N)$.\\

In this situation, it is therefore natural to extend \eqref{variationalproblem} to $u \in \bv(\Omega; \R^N)$. In doing so, we extend by lower semicontinuity, for instance, using the Lebesgue-Serrin extensions (see, for example, \cite{Serrin59}) and define:
\begin{equation}\label{eq:relaxedFunctional}
\overline{\mathscr F}_{u_0}[u; \Omega]:= \inf \left\{ \liminf_{j \to \infty} \int \limits_\Omega F(x, \D u_j) \dd{x}:\begin{array}{l} (u_j)_j \subset \sobo^{1,1}_{u_0}(\Omega; \R^N), \\[1mm] u_j \to u \text{ in } \lebe^1(\Omega; \R^N)\end{array} \right\}.
\end{equation}
If $F$ is convex with respect to the second variable, then the relaxed functional coincides with the original functional $\mathscr F$ on $\sobo^{1,1}_{u_0}(\Omega; \R^N)$. In addition, it is lower semicontinuous with respect to convergence in $\lebe^{1}(\Omega;\R^{N})$. A \textit{generalized minimizer} (or \textit{$\bv$ minimizer}) is a function $u \in \bv(\Omega; \R^N)$ such that
\begin{equation}\label{eq:gen_min}
  \overline{\mathscr F}_{u_0}[u; \Omega] \le \overline{\mathscr F}_{u_0}[v; \Omega] \ \text{for all } v \in \bv(\Omega; \R^N).  
\end{equation}
We denote the set of generalized minimizers by $\operatorname{GM}(\mathscr{F}; u_0)$. In this scenario, it is possible to represent the relaxed functional \eqref{eq:relaxedFunctional} in terms of integrals. The origins of such integral representations can be traced back to the works of \textsc{Goffman \& Serrin} \cite{GoffmanSerrin64} and \textsc{Reshetnyak} \cite{Reshetnyak68}, and even extend to the more general quasiconvex context; see, for example,  \textsc{Ambrosio \& Dal Maso} \cite{AmbrosioDalMaso92}, \textsc{Fonseca \& M\"{u}ller} \cite{FonsecaMüller93} and \textsc{Kristensen \& Rindler} \cite{KristensenRindler10}. Based on these results, the integral representation  in our case reads as
\begin{align}\label{eq:integralrepresentation}
    \begin{split} \overline{\mathscr F}_{u_0}[u; \Omega]  = \int \limits_\Omega F(x, \nabla u) \dd{x} & + \int \limits_\Omega F^\infty\left( x , \frac{\dd{\mathrm{D}^su}}{\dd{|\mathrm{D}^su|}}\right ) \dd{|\mathrm{D}^su|}  \\ & + \int_{\partial \Omega} F^\infty(x, (u_0-u) \otimes \nu_{\partial \Omega}) \dd{\mathscr H^{n-1}},\qquad u\in\bv(\Omega;\R^{N}).
    \end{split}
\end{align}
Here, the terms $u$ and $u_0$ in the boundary integral are understood in the sence of traces,  and the distributional gradient $\D u$ is split into its parts according to the Lebesgue-Radon-Nikodým theorem as follows:
$$\D u = \D^a u + \D^s u = \nabla u \mathscr L^n + \frac{\dd{\D^su}}{\dd{|\D^su|}}|\D^su|.$$
In order to grasp the behaviour of the integrands at infinity, which corresponds to $\D u$ becoming singular with respect to $\mathscr{L}^{n}$, the last ingredient in \eqref{eq:integralrepresentation} is given by the  \textit{recession function}
$$F^\infty(x,z) = \limsup_{t \to \infty} \frac{F(x,tz)}{t},\qquad x\in\overline{\Omega},\;z\in\R^{N\times n}.$$
For more on the space $\bv(\Omega; \R^N)$ and the properties of $\bv$-functions, we refer the reader to Chapter \ref{sec:IntroBV}. The existence of generalized minimizers of \eqref{variationalproblem} in $\bv(\Omega; \R^N)$ is by now well-known, see, for example,  \cite[Theorem 2.8]{Schmidt15} and essentially follows from an application of the direct method of calculus of variations once the functional is relaxed in the sense of \eqref{eq:relaxedFunctional}. However, we are interested in the regularity properties of generalized minimizers, which we will treat in the next section.

\subsection{Context and main results}\label{sec:context}
In the context of $\bv$ minimizers of \eqref{variationalproblem} it is desirable to obtain integrability and differentiability results that go beyond $u \in \bv(\Omega; \R^N)$. The regularity theory for linear growth problems has been a very active field of research over the past decades, and  we single out two particular types of theorems obtained so far: one class are Sobolev regularity results, stating that $u \in \sobo^{1,q}_{\mathrm{loc}}(\Omega; \R^N)$ for some $q \ge 1$ holds for every generalized minimizer. The second class of results is concerned with Hölder continuity of the gradients of minimizers.\\

So far, there is no available Sobolev regularity result for minimizers of non-autonomous convex linear growth integrals. We give an affirmative answer to this open problem with our following main result:
\begin{satz}[Sobolev regularity]\label{maintheorem}
Let $F: \Omega \times \R^{N\times n} \to \R$ with $F \in \hold(\bar{\Omega}\times\R^{N \times n})$ and $F(x,\cdot) \in \hold^2(\R^{N \times n})$ for every $x \in \Omega$. Assume that $F$ satisfies the following conditions:
    \begin{enumerate}[label = (\textbf{H\arabic*}) ]
        \item $F$ has linear growth: There exist constants $0<c_0 \le c_1<\infty$ and $c_2 \ge 0$ with
        $$c_0|z| -c_2 \le F(x,z) \le c_1(1+|z|)$$
        for all $z \in \R^{N\times n}$ and $x \in \Omega$. \label{cond:4} 
        \item There are $\lambda >0$ and $\mu >1$ such that
        $$\langle \D^2_z F(x,z) \xi, \xi  \rangle \geq \lambda \frac{|\xi|^2}{(1+|z|^2)^{\frac{\mu}{2}}}, \quad z, \xi \in \R^{N\times n}, x \in \Omega.$$\label{cond:2}
        \item There is a constant $C>0$ and a H\"{o}lder exponent $\alpha \in (0,1)$ such that
        $$|\D_z F(x_0,z) - \D_z F(x,z)| \leq C |x-x_0|^\alpha, \quad z \in \R^{N \times n}, x_0, x \in \Omega.$$\label{cond:3}
    \end{enumerate}
   Then, for $1 < \mu < 1+ \frac{\alpha}{n}$, any generalized minimizer $u \in \mathrm{GM}(\mathscr F; u_0)$ of \eqref{variationalproblem} is of class $\sobo^{1,p}_{\mathrm{loc}}(\Omega; \R^N)$ for all $1 \leq p < \frac{(3-\mu)n}{2 n-\alpha}$.
\end{satz}
For autonomous integrands $F: \R^{N \times n} \to \R$ only depending on the gradient variable, Sobolev regularity theorems of this kind are well-known and several other results for linear growth functionals are available, see, for example, \cite{BeckBulicekGmeineder20,BeEiFr24,BeckSchmidt13,Bildhauer02gradient_estimates,Bildhauer03book,Bildhauer03two-dimensional,Gmeineder2017,Gmeineder20BD,GmeinederKristensen19BD}.

The existing results usually rely on the assumption of $\mu$-ellipticity, meaning the existence of $\mu >1$ such that 
\begin{align}\label{eq:mu-ellipticity}
\lambda \frac{|\xi|^2}{\left(1+|z|^2\right)^{\frac{\mu}{2}}} \le \langle \D_z^2 F(x,z) \xi, \xi \rangle\le \Lambda \frac{|\xi|^2}{\left(1+|z|^2\right)^{\frac{1}{2}}}, \quad z,\xi \in \R^{N \times n}, \, x \in \Omega,
\end{align}
for some constants $0<\lambda \le \Lambda <\infty$. Integrands satisfying $\mu$-ellipticity are a special case of the class of $(s, \mu,q)$-growth integrands introduced by \textsc{Bildhauer, Fuchs \& Mingione} \cite{BildhauerFuchsMingione01}, see also \cite{BildhauerFuchs01}. Note that $(\mu=1)$-elliptic integrands cannot be of linear growth.\\

This assumption on the degenerate ellipticity allows to obtain a Sobolev regularity result (see \cite{BeckSchmidt13,Bildhauer02gradient_estimates,Bildhauer03book}) for autonomous linear growth integrands $F: \R^{N \times n} \to \R$. More precisely, if $F$ is $\mu$-elliptic with $1<\mu \le 3$, then any bounded generalized minimizer belongs to $\sobo^{1,1}_{\mathrm{loc}}(\Omega; \R^N)$. If one drops the boundedness assumption, the same result stays true in the range $1< \mu < 1+\frac 2n$. The main ideas for these Sobolev regularity results go back to \textsc{Bildhauer} \cite{Bildhauer02gradient_estimates,Bildhauer03book} who proved the requisite gradient integrability for one generalized minimizer. Due to the non-uniqueness of generalized minimizers (which occurs because of the lack of strict convexity of the recession function $F^\infty$), a refinement due to \textsc{Beck \& Schmidt} \cite{BeckSchmidt13} allows obtaining regularity for all generalized minimizers in the respective ranges of $\mu$.\\

A key observation is that the growth behavior of $\D_z^2F(x,z)$ exhibits asymmetry, with different rates from above and below. This feature is characteristic of linear growth integrands and structurally parallels the class of $(p,q)$-growth integrands. These are functions satisfying the following conditions: There exist constants such that $c_0, c_1, c_2 >0$ and $0< \lambda \le \Lambda < \infty$:
$$c_0 |z|^p -c_2 \le F(x,z) \le c_1(|z|^q+1), $$
$$\lambda (1+|z|^2)^{\frac{p-2}{2}} |\xi|^2 \le \langle \D_z^2F(x,z)\xi, \xi \rangle \le \Lambda (1+|z|^2)^{\frac{q-2}{2}} |\xi|^2$$
for all $z, \xi \in \R^{N \times n}, \ x \in \Omega$ and exponents $1<p<q < \infty$. The ratio $\frac{p}{q}$ encapsulates the asymmetry in growth and is conceptually analogous to the ellipticity parameter $\mu$. In both cases, there is a discrepancy between the functionals natural space of definition and the larger space needed for compactness arguments. For strictly convex $(p,q)$-growth integrands, compactness ensures the existence of a unique minimizer in $\sobo^{1,p}$ similar to how $\bv$ is the compactness space for linear growth integrands. The main goal, however, is to establish $\sobo^{1,q}_{\mathrm{loc}}$-regularity, as this is the natural space for the functional. This parallels the challenge of proving $\sobo^{1,1}_{\mathrm{loc}}$-regularity for linear growth integrands.

Important regularity results for minimizers of $(p,q)$-growth functionals were established by \textsc{Esposito, Leonetti, and Mingione}, see \cite{EspositoLeonettiMingione99NoDEA,EspositoLeonettiMingione99JDE,EspositoLeonettiMingione02} in the autonomous case and \cite{EspositoLeonettiMingione04} for the non-autonomous case. For non-autonomous, strictly convex functionals, they showed that $\sobo^{1,q}_{\mathrm{loc}}$-regularity holds when $\frac{q}{p} <1+ \frac{\alpha}{n}$, where $\alpha$ is the Hölder continuity exponent with respect to the first variable. Their results are sharp, as demonstrated by counterexamples for $\frac{q}{p} > 1+ \frac{\alpha}{n}$. They also proved that minimizers of the relaxed functional in $\sobo^{1,p}$ exhibit the $\sobo^{1,q}_{\mathrm{loc}}$ regularity. These proofs employ similar fractional Sobolev estimates and use of finite differences. However, the linear growth case poses additional challenges due to the potential non-uniqueness of $\bv$ minimizers. This makes a more refined stabilization procedure necessary (see Chapter \ref{sec:EkelandConstruction}).

The bound $1+ \frac{\alpha}{n}$ matches our upper bound for the ellipticity $\mu$ and shows that the presence of an $x$-dependence creates a gap in the possible parameters $p$ and $q$: For Lipschitz-continuous dependence on $x$, the condition is $\frac{q}{p} < 1+ \frac 1n$, while in the autonomous case the bound $\frac{q}{p} < 1+ \frac{2}{n}$ is sufficient for $\sobo^{1,q}_{\mathrm{loc}}$-regularity. This is entirely analogous to the linear growth setting described in Theorem \ref{maintheorem} in comparison to the autonomous case as discussed above.\\

Furthermore, this line of research extends beyond $(p,q)$-growth functionals. Functionals of nearly linear growth with similar conditions on the asymmetry in integrability exponents have been considered in the recent contributions by \textsc{De Filippis \& Mingione} \cite{DeFilippis1} and \textsc{De Filippis \& Piccinini} \cite{De_Filippis_2024}.\\

Having established higher integrability of the gradients of generalized minimizers, we turn to Hölder continuity results. One of the main results serving as the starting point for our considerations is the following partial regularity result which is due to \textsc{Anzellotti \& Giaquinta} \cite{AnzellottiGiaquinta88}:
\begin{prop}{\cite[Theorem 6.1]{AnzellottiGiaquinta88}}\label{prop:PartialRegularity}
    Let $F: \Omega \times \R^{N \times n} \to \R$ be an integrand of linear growth. Let $F(x, \cdot)$ be convex for $\mathscr L^n$-a.e. $x \in \Omega$ such that $\langle \D_z^2F(x,z)\xi, \xi\rangle \ge \lambda(z) |\xi|^2$ for all $z, \xi \in \R^{N \times n}$, $x \in \Omega$ and some function $\lambda(z)>0$. Assume furthermore that $|F(x, z) - F(y,z)| \le c |x-y|^\gamma (1+ |z|)$ for all $x,y \in \Omega$, $z \in \R^{N \times n}$ for some $\gamma \in (0,1)$. Then every generalized minimizer $u \in \operatorname{GM}(\mathscr{F}; u_0)$ of \eqref{variationalproblem} is partially regular in the following sense:
    There is an open set $\Omega_0 \subset \Omega$ with $\mathscr L^n(\Omega \backslash \Omega_0)=0$ such that $u\in \hold^{1, \alpha}_{\mathrm{loc}}(\Omega_0; \R^N)$ for some $\alpha \in (0,1)$. The set $\Omega_0$ is called the regular set of $u$.
\end{prop}
Regularity results of this category are called \textit{partial regularity} results (as opposed to \textit{full regularity}) because they do not establish regularity on the full set $\Omega$ but only on a subset $\Omega_0 \subset \Omega$. Apart from the linear growth setting with a convex integrand considered here, one can also consider \textit{quasiconvex} integrands $F$. These are functions $F: \Omega \times \R^{N \times n} \to \R$ satisfying 
$$F(x_0,z_0) \le \int \limits_{(0,1)^n}F(x_0,z_0 + \D \varphi) \dd{x} \quad \forall \varphi \in \sobo^{1,\infty}_0((0,1)^n; \R^{N \times n})$$
at every $x_0 \in \Omega$ and $z_0 \in \R^{N \times n}$. Partial regularity for quasiconvex linear-growth integrands was shown by \textsc{Gmeineder \& Kristensen} \cite{GmeinederKristensen19BVPartialRegularity}. On top of that, partial regularity results are also available for functionals depending on the symmetric part of the gradient or for functionals of non-standard growth, for example, $(p,q)$-growth or Orlicz growth; see upon others \cite{BärlinKessler2022,Gmeineder20BD,Gmeineder21symmquasiconvexBD,GmeinederKristensen2024pq,Li22,Schmidt14,Stephan24}. In the scalar-valued case $N=1$ or under strong structural assumptions there are also gradient Hölder continuity results available on the full set $\Omega$, see, for example,  \cite{BeckSchmidt15InteriorGradient,Bildhauer03book} and also the very recent contribution due to \textsc{De Filippis, De Filippis \& Piccinini} \cite{defilippis2024boundedminimizersdoublephase}. \\

In the situation of Proposition \ref{prop:PartialRegularity} we denote
$$\operatorname{Sing} u:= \Omega \backslash \Omega_0$$
the so-called singular set of $u$. It is known as a by-product from partial regularity proofs (see, for example, \cite{AnzellottiGiaquinta88,GmeinederKristensen19BVPartialRegularity,Schmidt09}) that the singular set of a $\sobo^{1,1}_{\mathrm{loc}}$-minimizer consists of the non-Lebesgue points of the derivative, i.e.
\begin{align}\label{eq:singularset}
\operatorname{Sing} u :=\left\{x_0 \in \Omega: \limsup_{\rho \searrow 0} \fint \limits_{\ball(x_0, \rho)}|\D u - (\D u)_{\ball(x_0, \rho)}|\ \dd x >0 \right \},
\end{align}
where for $x_0 \in \Omega$, the integral average is taken over the ball $\ball(x_0, \rho)$ for sufficiently small $\rho > 0$ such that $\ball(x_0, \rho) \subset \Omega$. Here, the description of the singular set differs from the one in \cite[Chapter 6]{AnzellottiGiaquinta88}, where additionally the points $x_0 \in \Omega$ for which
\begin{align}\label{eq:singular_set_u}
    \limsup_{\rho \searrow 0} \fint_{\ball(x_0, \rho)} |u- \bar u| \dd x>0 \quad \text{ for all } \bar u \in \R^N
\end{align}
are included in the singular set. This is not the case here, because assuming that \eqref{eq:singular_set_u} does not hold is only necessary to treat the additional $u$-dependence in \cite{AnzellottiGiaquinta88}.
\\

Although the singular set is $\mathscr L^n$-negligible, it could still be considerably large. An important question in the context of partial regularity is whether it is possible to provide a bound on the Hausdorff dimension $\dim_{\mathscr{H}}(\operatorname{Sing} u)$. Dimension reduction results for minimizers of variational integrals have been obtained in the case of $p$-growth,
$$\frac 1c |z|^p -d \le F(x, z) \le c(1+ |z|^p)$$
for some constants $c>0, d \ge 0$.
For $p \ge 2$, if $F$ is convex and $\D_zF(x,z)$ is differentiable with respect to the first variable and satisfies the growth condition 
$$|\D_x \D_zF(x,z)| \le C (1+|z|^2)^{\frac{p-1}{2}},$$ 
then $\dim_{\mathscr H}(\operatorname{Sing} u) \le n-2$ (see, for example, \cite{GiaquintaModica79} or \cite[Chapter 9]{Giusti}).
A dimension reduction even in the case where $F$ is non-differentiable with respect to $x$ is provided by \cite{Mingione03}.

Other results in this direction include, for example, \cite{KristensenMingione06,KristensenMingione07}. A survey on different dimension reduction results for the singular set can be found in \cite{mingione_regularity_2006,Mingione08wild_side}. \\

In view of partial regularity and dimension bounds, it is also the condition of $\mu$-ellipticity (see \eqref{eq:mu-ellipticity}) that leads to an improvement in comparison to $\mathscr L^n(\Omega \backslash \Omega_0)=0$ (see Proposition \ref{prop:PartialRegularity}). For autonomous integrands, the dimension bound
$$\dim_{\mathscr H}(\operatorname{Sing} u) \le n-1$$
is available in the range $1<\mu \le \frac{n}{n-1}$ if $n \ge 3$ and for $1<\mu<2$ if $n=2$. This result is in principle well-known, however, it is hard to find a proof in the literature. We will therefore include a sketch of the proof in Section \ref{sec:proofDimension}. Although partial regularity as recorded in Proposition \ref{prop:PartialRegularity} is available for non-autonomous integrands, a dimension bound is not known in that case.\\

Our second main result provides such a Hausdorff dimension bound, thereby extending the regularity theory for convex linear growth integrands to the non-autonomous case. It is a consequence of the higher integrability established in Theorem \ref{maintheorem}. 
\begin{satz}[Dimension bound]\label{maintheorem2}
    Let $F$ be as in Theorem \ref{maintheorem}. If $1 < \mu < \frac{3n}{3n-\alpha}$, then
    $$    \dim_{\mathscr{H}}(\operatorname{Sing} u) \leq n - \frac{\alpha}{2}$$
    holds for every generalized minimizer $u \in \operatorname{GM}(\mathscr{F}; u_0)$ of \eqref{variationalproblem}.
\end{satz}
Note in particular that the range for $\mu$ in Theorem \ref{maintheorem2} is more restrictive than in Theorem \ref{maintheorem}, since 
$
\frac{3n}{3n-\alpha} < 1 + \frac{\alpha}{n}.
$

\subsection{Main ideas of the proofs}
The proof of Theorem \ref{maintheorem} relies on a vanishing viscosity approximation strategy. The relaxation of the original functional \eqref{variationalproblem} to $\bv$ in the sense of \eqref{eq:relaxedFunctional} and \eqref{eq:integralrepresentation} leads to non-uniqueness of minimizers. This is due to the positive $1$-homogeneity of the recession function $F^\infty$ which results in a non-strictly convex term. Therefore, we start from a given generalized minimizer $u$ and construct a vanishing viscosity approximation sequence $(u_k)_k$ with the help of Ekeland's variational principle. This variational principle is a common tool in regularity results for linear growth functionals, see, for example, \cite{BeckSchmidt13,GmeinederKristensen19BD}. By using this approximating sequence, we avoid working directly with measures, which would appear as derivatives of $\bv$ functions. Instead, we are able to work in Sobolev spaces.\\

Taking an Euler-Lagrange-type inequality for the approximating sequence as the starting point, we then test with second order finite differences $\tau_{s,-h}(\rho^2 \tau_{s,h}u_k)$ in direction $s \in \{1,\dots, n\}$ of step width $h$. A major challenge is that we cannot differentiate the Euler-Lagrange equation or the corresponding Euler-Lagrange-type inequality due to the low regularity of $F$ in the first variable. The ideas presented here are similar to the strategy in \cite{Mingione03}, where the corresponding partial differential equations for non-autonomous $p$-growth functionals for $p \ge 2$ are considered. The observation is that we can interpret the assumption of Hölder continuity as some kind of fractional differentiability. Employing an auxiliary $V$-function and using the regularity assumptions on the integrand at hand, we obtain estimates on the decay of finite differences of the first derivatives as a power of $\abs{h}$. Thereby, we work within the scale of Nikolskiǐ spaces. In a certain range of the parameter $\mu$, embedding results then lead to higher integrability for the derivatives $\D u_k$. These carry over to the limit due to the construction of the approximating sequence as mentioned above.\\

Thereafter, Theorem \ref{maintheorem2} is a consequence of higher fractional Sobolev regularity and the measure density lemma. The fractional Sobolev regularity can be obtained from the same estimates on finite differences established before in the proof of Theorem \ref{maintheorem} and it gives an exact meaning to the idea that the Hölder assumption implies fractional differentiability. On the other hand, the measure density lemma provides an estimate on the set of non-Lebesgue points of functions in fractional Sobolev spaces.\\

\subsection{Organization of the paper}
In Section \ref{sec:prelims} we will provide the preliminaries and background material on functions of bounded variation and all relevant function spaces. We will also present the important measure density lemma that eventually leads to a dimension reduction. In Section \ref{sec:proofSobolev} we will prove Theorem \ref{maintheorem} by using the vanishing viscosity approximation. Section \ref{sharpness} is concerned with an example showing that in one dimension for $\mu> 1+ \alpha$, no Sobolev regularity for all generalized minimizers can be achieved. Thereby, it shows that the range for $\mu$ in Theorem \ref{maintheorem} is essentially sharp in one dimension. In Section \ref{sec:proofDimension} we extend this result to obtain fractional Sobolev regularity which leads to a dimension reduction. Possible generalizations and extensions of our results are discussed in Section \ref{sec:extensions}.

\section{Preliminaries}\label{sec:prelims}
\subsection{General Notation}
Throughout this paper, we assume $\Omega$ to be an open, bounded, and Lipschitz domain in $\R^n$. The Euclidean inner product on finite-dimensional real vector spaces is denoted by $\langle \cdot, \cdot \rangle$. The $n$-dimensional Lebesgue measure is denoted by $\mathscr{L}^n$ and the $(n-1)$-dimensional Hausdorff measure is denoted by $\mathscr{H}^{n-1}$. Given two positive, real valued functions $f,g$ we denote by $f \lesssim g$ that $f \leq Cg$ with a constant $C > 0$.

If $U \subset \R^n$ is measurable with $0 < \mathscr{L}^n(U) < \infty$ and $f \in \lebe^1(U;\R^N)$, we write as usual 
\begin{equation}\label{eq:mean-value-integral}
    (f)_{U} :=  \fint \limits_U f\dd{x} := \frac{1}{\mathscr{L}^n(U)} \int \limits_U f \dd{x}.
\end{equation}
For a given measurable map $f: \Omega \to \R^m$, a unit vector $e_s, \, s \in \{1,\dots,n\}$, and a step width $h \neq 0$, we define the finite difference $\tau_{s,h}f(x)$ by 
\begin{equation*}
    \tau_{s,h}f(x) := f(x + he_s) - f(x)
\end{equation*}
for all $x \in \Omega$ with $0<|h|<\mathrm{dist}(x, \, \partial \Omega)$. Moreover, for such $x$ we denote the difference quotient
\begin{equation*}
    \Delta_{s,h} f(x) := \frac{\tau_{s,h}f(x)}{h}.
\end{equation*}
Finally, for given $a \in \R^m, b \in \R^n$, we denote their dyadic product by $a \otimes b := ab^{\top}$.\\
\subsection{Functions of bounded variation}\label{sec:IntroBV}
We recall fundamental properties of functions of bounded variation, with primary references being \cite{AmbrosioFuscoPallara,evans_measure_1991}. A map $u \in \lebe^1(\Omega; \R^N)$ is called a function of bounded variation if the distributional derivative $\D u$ is a finite $\R^{N \times n}$-valued Radon measure. The space $\bv(\Omega; \R^N)$ of functions of bounded variation is a Banach space with the norm
$$\norm{u}_{\bv(\Omega; \R^N)}:= \norm{u}_{\lebe^1(\Omega; \R^N)} + |\D u|(\Omega),$$
where $|\D u|$ is the total variation measure of $\D u$. By the Radon-Nikodým theorem we can write $\D u$ as the sum of a measure that is absolutely continuous with respect to $\mathscr{L}^n$ and a part that is singular. We denote the density of the absolutely continuous part by $\nabla u$. For the singular part, we use the fact that any Radon measure is absolutely continuous with respect to its own total variation. Therefore, we can write
$$\D^su = \frac{\dd{\D^su}}{\dd{|\D^su|}}|\D^su|,$$
where as usual, $\frac{\dd{\D^su}}{\dd{|\D^su|}}$ denotes the density of $\D^su$ with respect to $|\D^su|$.\\

In many applications within the calculus of variations, the topology induced by this norm is overly restrictive. Consequently, alternative concepts of convergence in $\bv$ are crucial for addressing variational problems with linear growth.
\begin{itemize}
    \item \textit{Weak$^*$ convergence}: A sequence $(u_k)_k \subset \bv(\Omega; \R^N)$ is said to converge weakly$^*$ in $\bv$, denoted by $u_k \wstar u$, if $u_k \to u$ in $\lebe^1(\Omega; \R^N)$ and $\D u_k \wstar \D u$ as $\R^{N \times n}$-valued Radon measures.
    \item \textit{Strict convergence}: A sequence $(u_k)_k$ is said to converge strictly to $u$ if, in addition to weak$^*$ convergence, it holds that $|\D u_k|(\Omega) \to |\D u|(\Omega)$.
    \item \textit{Area-strict convergence}: A sequence $(u_k)_k$ is said to converge area-strictly to $u$ if in addition to weak$^*$ convergence the strict convergence of $(\mathscr L^n, \D u_k)$ to $(\mathscr L^n, \D u)$ as $\R^{1+N \times n}$-valued Radon measures holds. This is equivalent to
    $$\sqrt{1+ |\D u_k|^2}(\Omega) \to \sqrt{1+|\D u|^2}(\Omega)$$
    where the application of a convex function to a measure is defined using the so-called recession function. These functionals were first studied by \textsc{Goffman \& Serrin} \cite{GoffmanSerrin64} (for more details see, for example, \cite[Chapter 2.1]{Schmidt15} or \cite[Chapter 5]{AmbrosioFuscoPallara}).
\end{itemize}

\subsection{Application of Reshetnyak's Theorem}
In the context of linear growth functionals, it is necessary to deal with functionals of measures. The key result concerning (semi-)continuity of these functionals is due to \textsc{Reshetnyak} \cite{Reshetnyak68}. We present it here in the form of \cite[Theorem 2.4]{BeckSchmidt13}.
\begin{satz}[Reshetnyak]\label{thm:Reshetnyak}
    Let $(\mu_k)_k$ be a sequence of finite $\R^m$-valued Radon measures on $\Omega$ which converges weakly$^*$ to a finite $\R^m$-valued Radon measure $\mu$ on $\Omega$. Assume that all measures $\mu_k$ and $\mu$ take values in some closed convex cone $K$ in $\R^m$. 
    \begin{enumerate}
        \item If $f: K \to [0, \infty]$ is a lower semicontinuous, convex and $1$-homogeneous function, then
        $$\int \limits_\Omega f\left(\frac{\dd{\mu}}{\dd{|\mu|}}\right ) \dd{|\mu|} \le \liminf_{k \to \infty} \int \limits_\Omega f\left(\frac{\dd{\mu_k}}{\dd{|\mu_k|}}\right ) \dd{|\mu_k|}. $$
        \item If $\mu_k$ converges strictly to $\mu$ and $f: K \to [0, \infty)$ is continuous and $1$-homogeneous, then
        $$\int \limits_\Omega f\left(\frac{\dd{\mu}}{\dd{|\mu|}}\right ) \dd{|\mu|} = \lim_{k \to \infty} \int \limits_\Omega f\left(\frac{\dd{\mu_k}}{\dd{|\mu_k|}}\right ) \dd{|\mu_k|}.$$
    \end{enumerate}
\end{satz}
To see that this theorem is applicable for the relaxed functional as in \eqref{eq:integralrepresentation} in our case, one turns to the so-called perspective integrand $\overline{f}: \Omega \times [0, \infty) \times \R^{N \times n} \to [-\infty, \infty]$ defined by
$$\overline{f}(x,t,z):= \begin{cases}
    tf\left(x,\frac zt\right), \quad &\text{ for } t>0,\\
    f^\infty(x,z), \quad &\text{ for } t=0\\
\end{cases},$$
which is well-defined, lower semicontinuous and $1$-homogeneous if $f$ is convex. Then we have for $v \in \bv(\Omega; \R^N)$ with the decomposition $\D v = \nabla v \mathscr L^n + \frac{\dd{\D^sv}}{\dd{|\D^sv|}} |\D^sv|$:
\begin{equation}\label{eq:ReshetnyakApplication}
\int \limits_\Omega f(\nabla v) \dd{x} + \int \limits_\Omega f^\infty\left(\frac{\dd{\D^sv}}{\dd{|\D^sv|}}\right) \dd{|\D^sv|} = \int \limits_\Omega \overline{f}\left(x, \frac{\dd{(\mathscr L^n, \D v)}}{\dd{|(\mathscr L^n, \D v)|}}\right) \dd{|(\mathscr L^n, \D v)|}.
\end{equation}
The lower semicontinuity part may thus be applied to the functional on the left hand side for convex $f$ and a sequence $(v_k)_k \subset \bv(\Omega; \R^N)$ if $\D v_k \wstar \D v$ as Radon measures. The requirement of strict convergence in Theorem \ref{thm:Reshetnyak}, on the other hand, is equivalent to the area-strict convergence of the approximating sequence $( v_k)_k$ to $ v$ because the perspective integrand acts on $(\mathscr L^n, \D v)$. For more details, see, for example, \cite[Remark 2.5]{BeckSchmidt13}.\\

In the construction of an approximating sequence below in Section \ref{sec:proofSobolev}, we will use the following result due to \textsc{Bildhauer}. It allows us to approximate any $u \in \bv(\Omega; \R^N)$ area-strictly by smooth functions with arbitrarily prescribed traces. To this end, let $\tilde \Omega \Supset \Omega$ be a bounded Lipschitz domain and extend the given boundary datum $u_0 \in \sobo^{1,1}(\Omega; \R^N)$ to a function $\tilde u_0 \in \sobo^{1,1}_0(\tilde \Omega; \R^N)$. Then it holds
\begin{lemma}[{\cite[Lemma B.2]{Bildhauer03book}}]\label{lem:BildhauerApprox}
    Given $\tilde \Omega$ and $\tilde u_0$ as above, let $u \in \bv(\Omega; \R^N)$ and denote by $\tilde u$ its extension to $\tilde \Omega$ by $\tilde u_0$. There is a sequence $(u_k)_k \subset u_0 + \hold_c^\infty(\Omega; \R^N)$ such that $\tilde u_k$ converges to $\tilde u$ area-strictly, where $\tilde u_k$ denotes the extension of $u_k$ to $\tilde \Omega$ by $\tilde u_0$.
\end{lemma}
The area-strict convergence of this approximating sequence makes Theorem \ref{thm:Reshetnyak} applicable by the above remarks.

\subsection{Besov spaces and embeddings}
 In this section we review some important facts and results about fractional spaces. Firstly, we recall that, given $1 \leq p < \infty$ and $0 < \alpha < 1$, a measurable map $u: \Omega \to \R^N$ belongs to the fractional Sobolev space $\sobo^{\alpha,p}(\Omega; \R^N)$ if and only if $u \in \lebe^p(\Omega; \R^N)$ and the \textit{Gagliardo seminorm}
 \begin{equation*}
    [u]^p_{\sobo^{\alpha,p}(\Omega; \R^N)} := \iint \limits_{\Omega \times \Omega} \frac{|u(x) - u(y)|^p}{|x-y|^{n + \alpha p}} \dd{(x,y)} < \infty.
\end{equation*}
 of $u$ is finite. The norm on the fractional Sobolev space is given by
$$\norm{u}_{\sobo^{\alpha, p}(\Omega; \R^N)}:= \norm{u}_{\lebe^p(\Omega; \R^N)} +[u]_{\sobo^{\alpha,p}(\Omega; \R^N)}.$$
The fractional Sobolev spaces arise as special cases of the more general Besov spaces which we will recall next.\\

There are several equivalent ways to define Besov spaces. The most suitable for us is the following: Let $0<\theta<1$ and $1 \le p,q \le \infty$. Let $\Omega \subset \R^n$ be open and write $\Omega_h:= \{x \in \Omega: \mathrm{dist}(x, \partial \Omega)>h\}$. We define for $u \in \lebe^1_{\mathrm{loc}}(\Omega; \R^N)$ the quantities
\begin{align*}
[u]_{\operatorname{B}_{p,q}^\theta(\Omega; \R^N)}:=& \sum_{s=1}^n \left(\int \limits_0^\infty \left(\frac{\norm{\tau_{s,h}u}_{\lebe^p(\Omega_h; \R^N)}}{h^\theta} \right)^q \frac{\dd{h}}{h} \right)^{\frac 1q}, \quad &\text{ if }& 1 \le p,q <\infty,\\
  [u]_{\operatorname{B}_{p,\infty}^\theta(\Omega; \R^N)}:=& \sup_{h>0} \, \max_{s \in \{1, \dots, n\}} \frac{\norm{\tau_{s,h}u}_{\lebe^p(\Omega_h; \R^N)}}{h^\theta}, \quad &\text{ if }& 1\le p <\infty, q=\infty.
\end{align*}
The space of functions such that
$$\norm{u}_{\operatorname{B}_{p,q}^\theta(\Omega; \R^N)}:= \norm{u}_{\lebe^p(\Omega; \R^N)} + [u]_{\operatorname{B}_{p,q}^\theta(\Omega; \R^N)} < \infty$$
is then called the $(\theta, p,q)$-Besov space $\operatorname{B}_{p,q}^\theta(\Omega; \R^N)$. Note that $\operatorname{B}_{p,p}^\theta(\Omega; \R^N) \simeq \sobo^{\theta,p}(\Omega; \R^N)$ whenever $\Omega$ is a smooth domain (see \cite[Chapter 3.4.2]{Triebel}). We also call $\mathcal N^{\theta,p}:= \operatorname{B}_{p,\infty}^\theta$ the $(\theta,p)$-Nikolskiǐ space.\\

\subsection{The space $\sobo^{-1,1}(\Omega;\R^{N})$}\label{sec:negativeSobolevSpace}
We are going to apply Ekeland's variational principle (see below in Lemma \ref{lem:Ek_Prin}) in the negative Sobolev space $\sobo^{-1,1}(\Omega; \R^N)$. In contrast to negative Sobolev spaces $\sobo^{-k, p}(\Omega; \R^N)$ for $p>1$, this space can not be defined as the dual of a Sobolev space. Instead, it is defined directly by use of distributions.
\begin{defi}
    The space $\sobo^{-1,1}(\Omega; \R^N)$ is defined as the set of all $T \in \mathscr D'(\Omega; \R^N)$ such that there are $T_0, T_1, \dots, T_n \in \lebe^1(\Omega; \R^N)$ with
    $$T= T_0 + \sum_{j=1}^n \partial_j T_j$$
    in $\mathscr D'(\Omega; \R^N)$. We endow $\sobo^{-1,1}(\Omega; \R^N)$ with a norm via
    $$\norm{T}_{\sobo^{-1,1}(\Omega; \R^N)}:= \inf \left\{ \sum_{j=0}^n \norm{T_j}_{\lebe^1(\Omega; \R^N)} \right\}$$
    where the infimum ranges over all possible representations of $T$ as above.
\end{defi}
This definition gives rise to a Banach space, see \cite[Chapter 2]{BeckSchmidt13}.\\

Let us further remark that for a function $w \in \lebe^1(\Omega;\R^N)$, we can estimate the $\sobo^{-1,1}$-norm of its derivative and finite difference quotients. 
More precisely, for $s \in \{1, \dots, n\}$ and $h > 0$ there holds
\begin{align}
    \norm{\partial_s w}_{\sobo^{-1,1}(\Omega;\R^N)} &\leq \norm{w}_{\lebe^{1}(\Omega;\R^N)}, \notag\\
    \norm{\Delta_{s,h} w}_{\sobo^{-1,1}(\Omega_h;\R^N)} &\leq \norm{w}_{\lebe^{1}(\Omega;\R^N)}. \label{Dw-11}
\end{align}

The first inequality follows immediately from the definition of the norm on $\sobo^{-1,1}(\Omega; \R^N)$ while for the second inequality we use the representation
$$\Delta_{s,h}w(x) = \frac 1h \int \limits_0^1 \partial_s w(x+the_s) \cdot h \dd{t} = \frac{\partial}{\partial x_s} \int \limits_0^1 w(x+the_s)\dd{t}$$
for $w \in \operatorname{C}^1(\Omega; \R^N)$ which implies the claim in that case. By density the inequality follows for all $w \in \lebe^1(\Omega; \R^N)$.\\

This property illustrates the benefit of selecting the perturbations in Ekeland's variational principle to be in $\sobo^{-1,1}(\Omega; \R^N)$. Given our focus on fractional difference quotients of the derivatives of the functions in an approximating sequence, this approach enables us to estimate these fractional derivatives by a prescribed uniform bound on the $\sobo^{1,1}$-norm of the sequence. Similarly, it enables us to obtain estimates for second derivatives or second order difference quotients by 
\begin{equation}\label{eq:DifferenceQuotientSecondOrder}
\norm{\D^2 u}_{\sobo^{-1,1}(\Omega; \R^{N \times n^2})} \lesssim \norm{\D u}_{\lebe^1(\Omega; \R^{N\times n})} \lesssim \norm{u}_{\sobo^{1,1}(\Omega; \R^N)}.
\end{equation}

\subsection{The measure density lemma}
In obtaining a Hausdorff dimension bound for the singular set we will work with the measure density lemma originally due to \textsc{Giusti}, see for example, \cite[Prop. 2.7]{Giusti}. For the sake of completeness we present the result here as it is given in \cite[Chapter 1.4]{Beck}.
\begin{lemma}\label{lem:measuredensity}
Let $\Omega$ be an open set in $\R^n$ and let $\lambda$ be a finite, non-negative and non-decreasing set function on the family of open subsets of $\Omega$ which is countably super-additive, meaning that for any family $\{O_i\}_{i \in \N}$ of pairwise disjoint open subsets of $\Omega$ it holds that
$$\lambda \left( \bigcup_{i \in \N}O_i\right )\ge \sum_{i \in \N} \lambda(O_i).$$
Then, for every $\alpha \in (0,n)$ the sets
$$E^\alpha:= \{x_0 \in \Omega: \limsup_{\rho \searrow 0} \rho^{-\alpha} \lambda\left(\ball(x_0, \rho)\right)>0\}$$
satisfy $\dim_{\mathscr H}(E^\alpha) \le \alpha$.
\end{lemma}
With the help of this lemma, it is easy to obtain the following dimension bounds on the set of non-Lebesgue points of functions in fractional Sobolev spaces.
\begin{prop}[{\cite[Section 4]{Mingione03}, \cite[Prop. 1.76]{Beck}}]\label{thm:fracsobo-hausdorff}
    Let $f \in \sobo^{\theta, p}(\Omega; \R^N)$ for $0< \theta \le 1$ and $1 \le p <\infty$ with $\theta p <n$. Moreover, let
    $$\Omega_1:= \left \{x_0 \in \Omega: \limsup_{\rho \searrow 0} \fint \limits_{\ball(x_0, \rho)} \left| f(x) - (f)_{\ball(x_0, \rho)}\right |^p \dd{x} >0\right \},$$
    $$\Omega_2:= \left \{x_0 \in \Omega: \limsup_{\rho \searrow 0} \left | (f)_{\ball(x_0, \rho)}\right | = \infty \right \}.$$
Then we have
$$\dim_{\mathscr H}(\Omega_1) \le n -\theta p \quad \text{ and } \quad \dim_{\mathscr H}(\Omega_2) \le n -\theta p.$$
\end{prop}

In light of the characterization of the singular set of a minimizer given by \eqref{eq:singularset}, the primary objective in the proof of Theorem \ref{maintheorem2} is therefore to demonstrate that generalized minimizers of \eqref{variationalproblem} belong to fractional Sobolev spaces.

\section{Improved Sobolev regularity - Proof of Theorem \ref{maintheorem}}\label{sec:proofSobolev}
In this section, we will prove Theorem \ref{maintheorem}. The first step is to obtain a vanishing viscosity approximating sequence for any given generalized minimizer. Here, it is not sufficient to take a stabilized functional of the structure
$$\mathscr{F}_k[u; \Omega]:=\int \limits_\Omega \left (F(x, \D u) + \frac{1}{k^2} |\D u|^2 \right )\dd x$$
and to obtain estimates for the sequence $(u_k)_k$ of unique minimizers of $\mathscr F_k[-; \Omega]$. This strategy, which in the case of regularity for autonomous linear growth functionals, was used in the important works of \textsc{Bildhauer}, see, for example, \cite{Bildhauer02gradient_estimates,Bildhauer03book}, would only lead to estimates for one generalized minimizer; namely the limit of the sequence $(u_k)$ in a suitable topology. However, due to the term involving the recession function in \eqref{eq:integralrepresentation}, the functional at hand is no longer strictly convex, even though the integrand $F$ is. Therefore, generalized minimizers are not unique in general and it is necessary to start with any given minimizer $u$ and construct a corresponding approximating sequence with desirable regularity and convergence properties. This leads to the use of Ekeland's variational principle, see Lemma \ref{lem:Ek_Prin}.\\

Having constructed the approximating sequence $(u_k)$,we use the almost-minimality, given by Ekeland's variational principle, to obtain an Euler-Lagrange type inequality (Lemma \ref{Lemma 3.4}). Testing with second order finite differences of the functions $u_k$ itself, we are able to obtain fractional estimates for the derivatives $\D u_k$. Here, other than in the case of a $\hold^2$-dependence of $F$ on the first variable, we cannot differentiate the Euler-Lagrange equation of the system to obtain second order estimates. This also necessitates using finite differences instead of difference quotients.\\

Using the regularity of $F$ and a suitable auxiliary function, we end up with a Besov-Nikolskiǐ type estimate for the first derivatives. Embeddings into Lebesgue spaces then imply higher integrability in a suitable range of the ellipticity parameter $\mu$. Thereafter, the convergence properties of $(u_k)$, given from Ekeland's variational principle, make it possible to carry the integrability over to the limit $u$. \\

As we mentioned above, we will use the Ekeland variational principle from \cite{Ekeland1974}. We restate it here for the convenience of the reader.
\begin{lemma}\label{lem:Ek_Prin}
    Let $(X,d)$ be a complete metric space and let $F: X \to \R \cup \{+\infty\}$ be a lower semicontinuous functional which is not identically $+\infty$ and bounded from below. If for some $\eps>0$ and some $u \in X$ there holds
    $$F[u] \le \inf_X F +\eps,$$
    then there exists $v \in X$ such that
    \begin{align*}
    F[v] &\le F[u],\\
        d(u,v) &\le \sqrt \eps,\\
        F[v] &< F[w]+ \sqrt \eps d(v,w) \quad \textrm{for all } w \in X\backslash\{v\}.
    \end{align*}
\end{lemma}

In the following vanishing viscosity approximation we will stabilize the functional $\mathscr F[u]$ in a way that makes the next two auxiliary lemmas applicable.
\begin{lemma}[{\cite[Lemma 3.2]{GmeinederKristensen19BD}, \cite[Lemma 2.6]{BeckSchmidt13} }]\label{Lemma 3.2}
Let $F: \Omega \times \R^{N \times n} \to \R$ be a Carathéodory function such that $F(x, \cdot)$ is convex for almost all $x \in \Omega$ with
\begin{equation}
    c |z|^p - \theta \le F(x, z) \le \Theta(1+ |z|^p) \label{eq:coerc_bounds}
\end{equation}
for some $p>1$ and for all $z \in \R^{N \times n}$ with constants $c, \theta, \Theta >0$. Given a boundary datum $u_0 \in \sobo^{1,p}(\Omega; \R^N)$, the functional
$$\mathscr F[u]:=\begin{cases} \displaystyle
    \int \limits_\Omega F(x, \D u(x)) \dd{x}, \ &\mathrm{if } \ u \in \sobo^{1,p}_{u_0}(\Omega; \R^N),\\
    +\infty, &\mathrm{if } \ u \in (\sobo^{-1,1}\backslash \sobo^{1,p}_{u_0})(\Omega; \R^N)\\
\end{cases}$$
is lower semicontinuous with respect to norm convergence on $\sobo^{-1,1}(\Omega; \R^N)$.  
\end{lemma}
\begin{proof}
    The proof is similar to the one in \cite[Lemma 3.2]{GmeinederKristensen19BD}, but we will give it here for sake of completeness. Let $(u_k)_k \subset \sobo^{-1,1}(\Omega; \R^N)$ be a sequence with $u_k \to u$ in $\sobo^{-1,1}(\Omega; \R^N)$ for $k \to \infty$. If $ \liminf_{k \to \infty} \mathscr F[u_k]=+\infty$ there is nothing to show. Therefore we may assume that there is a subsequence $(u_k)_k$ (not relabeled), such that \[
    \lim \limits_{k \to \infty} \mathscr{F}[u_k] = \liminf \limits_{k \to \infty} \mathscr{F}[u_k] < \infty.
    \]
    Due to \eqref{eq:coerc_bounds} we deduce that $\D u_k$ is uniformly bounded in $\lebe^p$ and by Poincaré's inequality so is $u_k$, meaning we have a subsequence $u_{k_l}$ which converges weakly to some $v \in \sobo^{1,p}(\Omega; \R^N)$. Due to the continuity of the trace operator with respect to weak convergence in $\sobo^{1,p}(\Omega; \R^N)$ we even have $v \in \sobo_{u_0}^{1,p}(\Omega; \R^N)$. Since $\Omega$ is Lipschitz we have $u_{k_l} \to v$ strongly in $\lebe^p(\Omega; \R^N)$ by the compact embedding $\sobo^{1,p}(\Omega; \R^N) \hookrightarrow\hookrightarrow \lebe^p(\Omega; \R^N)$. We conclude $u = v$ $\mathscr{L}^n$-a.e. because of the embeddings
    $$\lebe^p(\Omega; \R^N) \hookrightarrow \lebe^1(\Omega; \R^N) \hookrightarrow \sobo^{-1,1}(\Omega; \R^N)$$
    on bounded domains. Now using the growth bound in \eqref{eq:coerc_bounds} and the convexity of $F$ in the second argument we can conclude by standard arguments, for example, \cite[Theorem 2.6]{Rindler} that $\mathscr{F}$ is lower semicontinuous with respect to norm convergence on $\sobo^{-1,1}(\Omega; \R^N)$. 
\end{proof}

The next lemma shows how a linear growth integrand can be modified in order to fit into the framework of Lemma \ref{Lemma 3.2}.

\begin{lemma}\label{Lemma 3.3}
Let $F: \Omega \times \R^{N \times n}\to \R$ be a function of linear growth in the sense of Assumption \ref{cond:4}. Then for any $\Theta>0$ there exist constants $C>0$ and $\vartheta>0$ such that
$$\Theta |z |^2 - \vartheta \le F(x,z) + \Theta |z|^2 \le C (1+|z|^2)$$
for all $z \in \R^{N\times n}, x \in \Omega$.
\end{lemma}
\begin{proof}
    By the linear growth hypothesis we have
    $$\Theta |z|^2 - c_2 \le c_0|z| - c_2 +\Theta|z|^2 \le F(x,z) + \Theta |z|^2 \le c_1(1+|z|) + \Theta |z|^2$$
    which immediately gives the claim by distinguishing the cases $|z|\le 1$ and $|z|>1$ on the right hand side.
\end{proof}

\subsection{Implementation of the Ekeland approximation}\label{sec:EkelandConstruction}
We closely follow the steps developed in \cite[Chapter 5.1]{BeckSchmidt13} and \cite[pp. 10-12]{GmeinederKristensen19BD} for the construction of a good approximating sequence. Here, the use of Ekeland's variational principle is necessary due to the possible non-uniqueness of generalized minimizers. This phenomenon occurs because of the relaxation of the functional (see \eqref{eq:relaxedFunctional}) even in the case where $F(x, \cdot)$ is strictly convex for all $x \in \Omega$. In constructing an approximating sequence for a given generalized minimizer $u$, we have to use perturbations that are weak enough, so that we can deal with them using the available a priori bounds (see also the comments in Section \ref{sec:negativeSobolevSpace}).\\ 

Let $u \in \vari(\Omega; \R^N)$ be a generalized minimizer of $\mathscr F$ (cf. equation \eqref{eq:gen_min}). From the approximation result Lemma \ref{lem:BildhauerApprox} and from the Reshetnyak continuity theorem (cf. Theorem \ref{thm:Reshetnyak}) we know that
$$ \inf_{u_0 + \sobo^{1,1}_0(\Omega; \R^N)} \mathscr F = \min_{\vari(\Omega; \R^N)} \overline{\mathscr F}_{u_0}.$$
We denote $ \mathscr D:= u_0 + \sobo^{1,1}_0(\Omega; \R^N)$. Again by Lemma \ref{lem:BildhauerApprox} we find a sequence $(w_k) \in \mathscr D$ with $w_k \to u$ in $\lebe^1(\Omega; \R^N)$ and $(\mathscr L^n, \D w_k) \to (\mathscr L^n, \D u)$ strictly as $k \to \infty$. Since all $w_k$ are in $\sobo^{1,1}(\Omega; \R^N)$ we obtain again by using the Reshetnyak continuity theorem that $\mathscr F(w_k) \to \inf \mathscr F[\mathscr D]$ as $k \to \infty$.
Therefore $(w_k)$ is a minimizing sequence for $\mathscr F$. Possibly passing to a non-relabeled subsequence we may conclude that
\begin{align}\label{eq:3.4}
\mathscr F[w_k] \le \inf \mathscr F[\mathscr D] + \frac{1}{8k^2}\end{align}
for all $k \in \mathbb N$. By the $\mu$-ellipticity (and hence convexity) and the linear growth of $F$, we know that $F$ is Lipschitz with some constant $L>0$. For the Dirichlet data, we find a sequence $(u_k^{\partial \Omega}) \subset \sobo^{1,2}(\Omega; \R^N)$ which satisfies
\begin{align}\label{eq:3.5}
\norm{ u_k^{\partial \Omega} -u_0 }_{\sobo^{1,1}(\Omega; \R^N)} \le \frac{1}{8Lk^2}
\end{align}
for all $k \in \mathbb N$. We denote $\mathscr{D}_k:= u_k^{\partial \Omega}+ \sobo^{1,2}_0(\Omega; \R^N)$ and find that, because of $w_k \in u_0 + \sobo^{1,1}_0(\Omega; \R^N)$ there exists a sequence $(v_k) \subset \mathscr D_k$ with
$$\norm{(v_k - u_k^{\partial \Omega})- (w_k - u_0)}_{\sobo^{1,1}(\Omega;\R^N)}\le \frac{1}{8Lk^2}$$
from which we infer
\begin{align}\label{eq:3.6}
    \norm{v_k - w_k}_{\sobo^{1,1}(\Omega; \R^N)} \leq \frac{1}{4Lk^2}
\end{align}
for all $k \in \mathbb N$ by the triangle inequality.
For all $\psi \in \sobo_0^{1,2}(\Omega; \R^N)$ it holds 
\begin{align*}
    \inf_{u_0+\sobo_{0}^{1,2}(\Omega; \R^N)} \mathscr{F} & \leq \mathscr{F} [u_0 + \psi] \nonumber\\ 
    &= \left( \int \limits_\Omega \left[  F(x, \D(u_0 + \psi)) - F(x, \D (u_k^{\partial \Omega} + \psi)) \right] \dd{x} \right) + \int \limits_\Omega F(x, \D (u_k^{\partial \Omega} + \psi)) \dd{x}\\
    & \leq L \int \limits_\Omega |\D (u_0 - u_k^{\partial \Omega})| \dd{x} + \int \limits_\Omega F(x, \D (u_k^{\partial \Omega} + \psi)) \dd{x} \quad \text{ (Lipschitz continuity of  $F$)}\\
    & \leq L \norm{u_k^{\partial \Omega} - u_0}_{\sobo^{1,1}(\Omega; \R^N)} + \int \limits_\Omega F(x, \D (u_k^{\partial \Omega} + \psi)) \dd{x} \nonumber\\
    & \overset{\eqref{eq:3.5}}{\leq} \frac{1}{8k^2} + \int \limits_\Omega F(x, \D (u_k^{\partial \Omega} + \psi)) \dd{x}.
\end{align*}
Taking the infimum over $\psi \in \sobo^{1,2}_0(\Omega; \R^N)$ and using the fact that $u_0 \,+\, \sobo^{1,2}_0(\Omega; \R^N)$ is norm-dense in $\mathscr D = u_0 + \sobo_0^{1,1}(\Omega;\R^N)$ gives
\begin{equation}\label{eq:FDDk}
  \inf_{\mathscr D = u_0 + \sobo_0^{1,1}(\Omega;\R^N)} \mathscr F = \inf_{u_0 + \sobo^{1,2}_0(\Omega, \R^N)} \mathscr F \le \inf_{\mathscr D_k = u_k^{\partial \Omega} + \sobo^{1,2}_0(\Omega;\R^N)} \mathscr F + \frac{1}{8k^2}.  
\end{equation}
Moreover, by using \eqref{eq:3.6} and the Lipschitz bound on $F$ we get 
\begin{equation}\label{eq:lip_Fvw}
    \mathscr F [v_k] - \mathscr F [w_k] = \int \limits_\Omega \left[  F(x, \D v_k) - F(x, \D w_k) \right] \dd{x} \leq L \int \limits_\Omega |\D (v_k - w_k)| \dd{x} \leq \frac{1}{4k^2}.
\end{equation}
Furthermore by using \eqref{eq:3.4}, \eqref{eq:FDDk} and \eqref{eq:lip_Fvw}, we obtain
\begin{align}\label{eq:3.8}
\mathscr F [v_k] \overset{\eqref{eq:lip_Fvw}}{\leq} \mathscr F [w_k] + \frac{1}{4k^2} \overset{\eqref{eq:3.4}}{\leq} \inf \mathscr F[\mathscr D] + \frac{3}{8k^2} \overset{\eqref{eq:FDDk}}{\leq} \inf \mathscr F[\mathscr D_k] + \frac{1}{2k^2}
\end{align}
for all $k \in \mathbb N$.\\

Now we define integrands $F_k$, which are regularized versions of $F$. Let 
\begin{equation}\label{eq:3.9}
    \begin{aligned}
  F_k(x, z) &:= F(x, z) + \frac{1}{2k^2A_k}(1+|z|^2),\\ 
 \text{ where} \quad A_k &:= 1 + \int \limits_\Omega (1+|\D v_k|^2) \dd{x}.  
\end{aligned}
\end{equation}
With these functions we can define
\begin{equation}\label{def:scrF_k}
 \mathscr F_k [w] :=  \begin{cases} \displaystyle
    \int \limits_\Omega F_k(x, \D w) \dd{x}, \ &\text{provided } w \in \mathscr D_k,\\
    + \infty, \ &\text{provided } w \in \sobo^{-1,1}(\Omega; \R^N)\backslash \mathscr D_k.\\
\end{cases}   
\end{equation}
We aim to apply Lemma \ref{Lemma 3.2} for each $k \in \mathbb N$ and for $p = 2$. We note that the conditions specified in Lemma \ref{Lemma 3.2} are met as a consequence of Lemma \ref{Lemma 3.3}. Hence, all $\mathscr F_k$ are lower semicontinuous with respect to norm convergence in $\sobo^{-1,1}(\Omega; \R^N)$. Now we compute for $v_k \in \mathscr D_k$
\begin{align}\label{eq:explanationF_kF}
    \mathscr F_k[v_k] = \int \limits_{\Omega} \left[  F(x, \D v_k) + \frac{1}{2k^2A_k}(1+| \D v_k |^2) \right] \dd{x} \overset{\eqref{eq:3.9}_2}{\leq} \mathscr F [v_k] + \frac{1}{2k^2}.
\end{align}
Then by \eqref{eq:3.8}, we further obtain
\begin{equation*}
    \mathscr F [v_k] + \frac{1}{2k^2} \le \inf \mathscr F[\mathscr D_k] + \frac{1}{k^2}.
\end{equation*}
Observe that by the definition \eqref{def:scrF_k} of $\mathscr F_k$, we have $\inf \mathscr F[\mathscr D_k] \leq \inf \mathscr F_k[\sobo^{-1,1}(\Omega; \R^N)]$. By combining all these ineequalities, we arrive at
\begin{align}\label{eq:3.10}
\mathscr F_k [v_k] \overset{\eqref{eq:explanationF_kF}}{\leq} \mathscr F [v_k] + \frac{1}{2k^2} \overset{\eqref{eq:3.8}}{\leq} \inf \mathscr F[\mathscr D_k] + \frac{1}{k^2} \le \inf \mathscr F_k[\sobo^{-1,1}(\Omega; \R^N)] + \frac{1}{k^2}.
\end{align}
This establishes that $v_k$ is an almost-minimizer of the functional $\mathscr F_k$. We are now in place to apply Ekeland's variational principle (cf. Lemma \ref{lem:Ek_Prin}) in the Banach space $\sobo^{-1,1}(\Omega; \R^N)$ to obtain a sequence $(u_k) \subset \sobo^{-1,1}(\Omega; \R^N)$ such that
\begin{align}\label{eq:3.11}
\lVert u_k - v_k \rVert_{\sobo^{-1,1}(\Omega; \R^N)} &\le \frac 1k, \notag \\ 
\mathscr F_k [u_k] &\le \mathscr F_k [w] + \frac 1k \lVert w - u_k \rVert_{\sobo^{-1,1}(\Omega; \R^N)} \quad  \forall w \in \sobo^{-1,1}(\Omega; \R^N), \ k \in \mathbb N.
\end{align}
We apply the second part of \eqref{eq:3.11} to $w=v_k$ to obtain
\begin{align*}
    \int \limits_\Omega |\D u_k| \dd{x} & \overset{\ref{cond:4}}{\leq} \frac{1}{c_0} (\mathscr{F}[u_k] + c_2\mathscr L^n(\Omega) ) \leq \frac{1}{c_0} (\mathscr{F}_k[u_k] + c_2\mathscr L^n(\Omega)) \\
    & \overset{\eqref{eq:3.11}_2}{\leq} \frac{1}{c_0} (\mathscr{F}_k[v_k] + \frac{1}{k} \norm{v_k - u_k}_{\sobo^{-1,1}(\Omega, \R^N)} + c_2\mathscr L^n(\Omega)) \\
    & \overset{\eqref{eq:3.11}_1}{\leq} \frac{1}{c_0} \left(\mathscr{F}_k [v_k] + \frac{1}{k^2} + c_2\mathscr L^n(\Omega) \right) \\
    & \overset{\eqref{eq:3.10}}{\leq} \frac{1}{c_0} \left(\inf \mathscr{F}_k [\sobo^{-1,1}(\Omega; \R^N)] + \frac{2}{k^2} + c_2\mathscr L^n(\Omega) \right) \leq C,
\end{align*}
where $C > 0$ is a finite constant independent of $k \in \N$. To see this, we use that $\inf \mathscr F_k[\sobo^{-1,1}(\Omega; \R^N)] <\infty$ which follows from the fact that clearly $\inf \mathscr F_k[\mathscr D_k] < \infty$ and that $\mathscr F_k$ is non-increasing in $k$.\\

In particular we can deduce that $(u_k)$ is uniformly bounded in $\sobo^{1,1}(\Omega; \R^N)$ because we have by the Poincaré inequality
\begin{align*}
    \int \limits_\Omega |u_k| \dd{x} \le \int \limits_\Omega \left(|u_k -u_k^{\partial \Omega}| + |u_k^{\partial \Omega}| \right)\dd{x} \le \int \limits_\Omega \left(|\D u_k| + |\D u_k^{\partial \Omega}| \right)\dd{x} + \int \limits_\Omega |u_k^{\partial \Omega}| \dd{x} \le C
\end{align*}
where the uniform boundedness of $u_k^{\partial \Omega}$ in $\sobo^{1,1}(\Omega; \R^N)$ follows from \eqref{eq:3.5}. Hence both $|u_k|$ and $|\D u_k|$ are uniformly $\lebe^1$-bounded and we have shown
\begin{equation}\label{eq:3.13}
    u_k \text{ is uniformly bounded in } \sobo^{1,1}(\Omega; \R^N).
\end{equation}
The functions $u_k$ are not necessarily minimizers of the modified functional $\mathscr F_k$, but they satisfy an Euler-Lagrange type inequality.
\begin{lemma}[Approximate Euler-Lagrange]\label{Lemma 3.4}
Let $F_k$ and $u_k$ be defined as above. For all $k \in \mathbb N$ we have
    \begin{equation}\label{eq:3.14}
        \left| \int \limits_\Omega \langle \D_z F_k(x,\D u_k), \D\varphi \rangle \dd{x} \right| \leq \frac{1}{k} \norm{\varphi}_{\sobo^{-1,1}(\Omega; \R^N)}
    \end{equation}
for all $\varphi \in \sobo^{1,2}_0(\Omega; \R^N)$.
\end{lemma}
\begin{proof}
    Fix $k \in \mathbb N$ and let $\varphi \in \sobo^{1,2}_0(\Omega; \R^N)$ be arbitrary. For every $\varepsilon>0$, by $u_k \in \mathscr D_k$ we have $u_k \pm \varepsilon \varphi \in \mathscr D_k$. Using the second line of \eqref{eq:3.11} we get
    $$\int \limits_\Omega \big (F_k(x, \D u_k) - F_k(x, \D u_k \pm \eps \D\varphi)\big ) \dd{x} = \mathscr F_k[u_k] - \mathscr F_k [u_k \pm \varepsilon \varphi] \le \frac{\varepsilon}{k} \norm{\varphi}_{\sobo^{-1,1}(\Omega, \R^N)}.$$
    Dividing by $\varepsilon >0$ and letting $\varepsilon \searrow 0$ we arrive at the claim.
\end{proof}

\begin{remark}
    We briefly comment here on the advantages of this specific way of obtaining an approximating sequence.
    
    Firstly, the Euler-Lagrange-type inequality \eqref{eq:3.14} can now be used to obtain estimates for fractional difference quotients. This is due to the $\sobo^{-1,1}$-norm on the right hand side, which enables us to use the uniform boundedness of the sequence $(u_k)$ in $\sobo^{1,1}(\Omega; \R^N)$ and the remarks on the norm of $\sobo^{-1,1}(\Omega; \R^N)$ in Section \ref{sec:negativeSobolevSpace} to control the right hand side of \eqref{eq:3.14} uniformly in $k$ upon testing with second order finite differences of $u_k$.
    
    Secondly, by construction we have $u_k \to u$ in $\sobo^{-1,1}(\Omega; \R^N)$. From \eqref{eq:3.13} we obtain a subsequence of $(u_k)_k$ converging weakly$^*$ in $\bv(\Omega; \R^N)$ and the convergence in $\sobo^{-1,1}(\Omega; \R^N)$ identifies $u$ as the weak$^*$ limit. This entails especially the weak$^*$ convergence of $\D u_k$ to $\D u$ as finite Radon measures. This leads to the limit $u$ inheriting suitable uniform estimates on the $u_k$ at the end of the proof of Theorem \ref{maintheorem}. Since we started from an arbitrary generalized minimizer, we will thereby obtain a result for every $u \in \operatorname{GM}(\mathscr{F}; u_0)$.
\end{remark}

\subsection{Fractional estimates for the approximating sequence}
We now proceed with the proof of Theorem \ref{maintheorem}. Fix an arbitrary generalized minimizer $u \in \operatorname{GM}(\mathscr{F}; u_0)$ and consider the sequence $(u_k)$ just constructed, cf. \eqref{eq:3.11}. For every $k \in \mathbb{N}$, the function $u_k$ satisfies the Euler-Lagrange type inequality \eqref{eq:3.14}. 

Given $x_0 \in \Omega$ and radii $0 < r < R < \text{dist}(x_0, \partial \Omega)$, we select a cut-off function $\rho \in \hold^1_c(\Omega; [0, 1])$ such that $\mathbbm{1}_{\ball(x_0, r)} \leq \rho \leq \mathbbm{1}_{\ball(x_0, R)}$. 

Since the regularity assertion of Theorem \ref{maintheorem} is local, we may assume without loss of generality that $\Omega$ is connected. For $0 < |h| < \text{dist}(\partial \ball(x_0, R); \partial \Omega)$ and $s \in \{1, \ldots, n\}$, define $\varphi$ by
\[
\varphi := \tau_{s, -h}(\rho^2 \tau_{s, h}(u_k)) \in \sobo_0^{1,2}(\Omega; \mathbb{R}^N),
\]
which will serve as a test function in the approximate Euler-Lagrange inequality \eqref{eq:3.14}. 
Since the differential operator $\D$ and the translation operator $\tau_{s,h}$ commute, we can write
\begin{equation}\label{eq:3.20}
\left| \int \limits_{\Omega} \left\langle \D_z F_k(x, \D u_k), \tau_{s, -h}(\D (\rho^2 \tau_{s, h} u_k)) \right\rangle \, \mathrm{d}x \right| \leq \frac{1}{k} \norm{\tau_{s, -h}(\rho^2 \tau_{s, h} u_k)}_{\sobo^{-1,1}(\Omega; \mathbb{R}^N)}.
\end{equation}
Using $A_k$ from the definition in \eqref{eq:3.9}, we apply discrete integration by parts to \eqref{eq:3.20} and use the product rule to derive the following inequality:
\begin{align}
    \mathit{I} &:= \int \limits_{\Omega} \langle \tau_{s,h} \D_z F_k(x,\D u_k), \rho^2 \tau_{s,h} \D u_k \rangle \dd{x} \nonumber\\
    &\leq \left| \int \limits_{\Omega} \langle \tau_{s,h}  \D_z F_k(x,\D u_k), 2 \rho \D \rho \otimes \tau_{s,h} u_k \rangle \dd{x} \right| +  \frac{1}{k} \norm{\tau_{s,-h}(\rho^2 \tau_{s,h}u_k)}_{\sobo^{-1,1}(\Omega; \R^N)} \nonumber \\
    &\leq \left| \int \limits_{\Omega} \langle \D_z F_k(x + he_s,\D u_k(x + he_s)) - \D_z F_k(x,\D u_k(x + he_s)), 2 \rho \D \rho \otimes \tau_{s,h} u_k \rangle \dd{x} \right| \nonumber\\
    & \hspace{1cm} + \left| \int \limits_{\Omega} \langle \D_z F_k(x,\D u_k(x + he_s)) - \D_z F_k(x,\D u_k(x)), 2 \rho \D \rho \otimes \tau_{s,h} u_k \rangle \dd{x} \right| \nonumber \\
    &\hspace{3cm} + \frac{1}{k} \norm{\tau_{s,-h}(\rho^2 \tau_{s,h}u_k)}_{\sobo^{-1,1}(\Omega; \R^N)} \nonumber \\ 
    &\leq \left| \int \limits_{\Omega} \langle \D_z F(x + he_s,\D u_k(x + he_s)) - \D_z F(x,\D u_k(x + he_s)), 2 \rho \D \rho \otimes \tau_{s,h} u_k \rangle \dd{x} \right| \nonumber\\
    & \hspace{1cm} + \left| \int \limits_{\Omega} \langle \D_z F(x,\D u_k(x + he_s)) - \D_z F(x,\D u_k(x)), 2 \rho \D \rho \otimes \tau_{s,h} u_k \rangle \dd{x} \right| \nonumber \\
    &\hspace{1.5cm} +  \frac{1}{A_k k^2} \left| \int \limits_{\Omega} \langle \tau_{s,h} \D u_k, 2 \rho \D \rho \otimes \tau_{s,h} u_k \rangle \dd{x} \right|+ \frac{1}{k} \norm{\tau_{s,-h}(\rho^2 \tau_{s,h}u_k)}_{\sobo^{-1,1}(\Omega; \R^N)}  \nonumber \\ 
    &=: \mathit{II} + \mathit{III} + \mathit{IV} + \mathit{V}. \label{sum_of_four}
\end{align}

Now we will bound the terms individually from above:
\begin{itemize}
    \item On $\mathit{II}$: Using the H\"{o}lder continuity of $\D_z F(\cdot, z)$ (see Condition \ref{cond:3}), we obtain
\begin{equation*}
    \mathit{II} \leq  C|h|^{\alpha}\ \int \limits_{\Omega} |\rho \D \rho \otimes \tau_{s,h} u_k| \, \mathrm{d}x  \leq C|h|^{\alpha}   \int \limits_{\ball(x_0,R)} |\D \rho \otimes \tau_{s,h} u_k| \, \mathrm{d}x.
\end{equation*}
Now by standard theory of difference quotients in $\sobo^{1,1}(\Omega; \R^N)$ (see also \cite[Lemma 2.2]{Mingione03}) we obtain
\begin{equation*}
    \int \limits_{\ball(x_0,R)} | \tau_{s,h} u_k| \, \mathrm{d}x  \leq C |h| \int \limits_{\ball(x_0,R')} |Du_k| \, \mathrm{d}x \leq C |h|.
\end{equation*}
Therefore we have
\begin{equation*}
    \mathit{II} \leq C |h|^{\alpha + 1}.
\end{equation*}
\item On $\mathit{III}$: Due to the convexity and the linear growth of $F$, the function $F$ is Lipschitz continuous with respect to the second variable for any fixed $x_0 \in \Omega$. Using \ref{cond:3}, we obtain for any $x \in \Omega, z \in \R^{N\times n}$:
$$|\D_zF(x,z)|\le |\D_zF(x,z)-\D_zF(x_0, z)|+ |D_zF(x_0, z)| \le C \mathrm{diam}(\Omega)^\alpha + C \le \tilde C.$$
Hence $|\D_zF(x,z)|$ is uniformly bounded independent of $x$ and $z$ and we may estimate
    \begin{equation*}
        \left|  \D_z F(x,\D u_k(x + he_s)) - \D_z F(x,\D u_k(x)) \right| \leq M \quad \text{for some } M > 0,
    \end{equation*}
    with $M$ independent of $x$ and $k$. Consequently, by standard estimates on difference quotients for functions in $\sobo^{1,1}$, we obtain
    \begin{align*}
       \mathit{III} 
       \leq CM |h| \int \limits_{\ball(x_0, R)} |\Delta_{s,h} u_k| \, \mathrm{d}x
       \leq C M |h| \norm{u_k}_{\sobo^{1,1}(\Omega, \mathbb{R}^N)}.
    \end{align*}
    Thus, we conclude
    \begin{equation*}
       \mathit{III} \leq C M |h|,
    \end{equation*}
    since $u_k$ is uniformly bounded in $\sobo^{1,1}(\Omega; \mathbb{R}^N)$.
    \item On $\mathit{IV}$: Applying the Cauchy-Schwarz and Young inequality, we find for $\delta > 0$ sufficiently small,
    \begin{align}
        \mathit{IV} &\leq \frac{\delta}{A_k k^2} \int \limits_{\Omega} |\rho \tau_{s,h} \D u_k|^2 \dd{x} + \frac{C(\delta)}{A_k k^2} \int \limits_{\Omega} |\D \rho \otimes \tau_{s,h} u_k|^2 \dd{x} \nonumber \\
        &= \frac{\delta}{A_k k^2} \int \limits_{\Omega} |\rho \tau_{s,h} \D u_k|^2 \dd{x} + \frac{C(\delta,\rho)|h|^2}{A_k k^2} \int \limits_{\ball(x_0,R)} |\Delta_{s,h} u_k|^2 \dd{x} \nonumber \\
        &\leq \frac{\delta}{A_k k^2} \int \limits_{\Omega} |\rho \tau_{s,h} \D u_k|^2 \dd{x} + \frac{C(\delta,\rho)|h|^2}{A_k k^2} \int \limits_{\Omega} |\partial_s u_k|^2 \dd{x} . \nonumber
    \end{align}

    Proceeding from here, we apply the definition of $\mathscr{F}_k$ and \eqref{eq:3.11}\textsubscript{2} with $w = v_k$ to derive
    \begin{align}\label{eq:BoundForIV}
        \frac{1}{2k^2A_k} \int \limits_{\Omega} |\partial_s u_k|^2 \dd{x} &\leq \frac{1}{2k^2A_k} \int \limits_{\Omega} (1 + |\D u_k|^2) \dd{x}\nonumber \\
        &\leq c(\Omega) \left (\mathscr{F}_k[u_k]+c_2\mathscr L^n(\Omega) \right )\nonumber \\
        &\overset{\eqref{eq:3.11}_2}{\leq} \left(\mathscr{F}_k[v_k] + \frac{1}{k} \norm{v_k-u_k}_{\sobo^{-1,1}(\Omega; \R^N)}+c_2\mathscr L^n(\Omega) \right)\nonumber \\
        &\leq C(\Omega) < \infty
    \end{align}
    with the constant $c_2 \ge 0$ from the linear growth assumption \ref{cond:4}. Consequently, we obtain
    \[
    \mathit{IV} \leq \frac{\delta}{A_k k^2} \int \limits_\Omega |\rho \tau_{s,h} \D u_k|^2  \dd{x} + C(\delta, \rho, \Omega)|h|^2 =: \mathit{IV}_1^{(\delta)} + \mathit{IV}_2^{(\delta)}.
    \]
    \item On $\mathit{V}$: We estimate using standard difference quotient estimates, equations \eqref{Dw-11}, \eqref{eq:DifferenceQuotientSecondOrder}, as well as the definition of $\rho$ and the uniform boundedness of $u_k$ in $\sobo^{1,1}(\Omega;\R^N)$:
\begin{align}
    \mathit{V} &= \frac{1}{k} \norm{\tau_{s, -h} (\rho^2 \tau_{s,h} u_k)}_{\sobo^{-1,1}(\Omega; \R^N)} = \frac{|h|^{2}}{k} \norm{\Delta_{s,-h} \left(\rho^2 \Delta_{s,h} u_k \right)}_{\sobo^{-1,1}(\Omega_h; \R^N)} \notag \\
    &\leq \frac{|h|^{2}}{k} \norm{\rho^2 \Delta_{s,h} u_k}_{\lebe^1(\Omega; \R^N)} \notag \\ 
    &\leq C \frac{|h|^{2}}{k} \norm{u_k}_{\sobo^{1,1}(\Omega; \R^N)} \leq C \frac{|h|^{2}}{k}. \label{3.26}
\end{align}
\end{itemize}

Now, we proceed to establish a lower bound for $\mathit{I}$. To achieve this, let $\mu$ be the ellipticity exponent from the theorem. We derive estimates by leveraging the H\"{o}lder continuity of $\D_z F(\cdot, z)$ and the definition of $F_k$:
\begin{align*}
    \mathit{I} &= \int \limits_{\Omega} \langle \tau_{s,h} \D_z F_k(x,\D u_k), \rho^2 \tau_{s,h} \D u_k \rangle \dd{x} \nonumber\\
    &= \int \limits_{\Omega} \langle \D_z F_k(x + he_s,\D u_k(x + he_s)) - \D_z F_k(x,\D u_k(x + he_s)), \rho^2 \tau_{s,h} \D u_k \rangle \dd{x} \nonumber\\
    & \hspace{1cm} +\int \limits_{\Omega} \langle \D_z F_k(x,\D u_k(x + he_s)) - \D_z F_k(x,\D u_k(x)), \rho^2 \tau_{s,h} \D u_k \rangle \dd{x} \nonumber \\ 
    &\geq -C|h|^\alpha \int \limits_\Omega\left | \rho^2 \tau_{s,h} \D u_k\right | \dd{x}  + \int \limits_{\Omega} \int \limits_0^1 \langle \D^2_z F(x,\D u_k +
    t \tau_{s,h} \D u_k)\rho \tau_{s,h}\D u_k, \rho \tau_{s,h} \D u_k \rangle \dd{t} \dd{x} \nonumber\\
    &\hspace{4cm} + \frac{1}{A_k k^2} \int \limits_{\Omega} |\rho \tau_{s,h} \D u_k|^2 \dd{x} =: \mathit{I'}.\nonumber
\end{align*}
Using the uniform $\sobo^{1,1}$-boundedness of $(u_k)$, the term $\int \limits_\Omega |\rho^2 \tau_{s,h}\D u_k| \dd x$ is uniformly bounded. By making use of the $\mu$-ellipticity for the second term, we conclude
\begin{align}
    \mathit{I'} \ge - C |h|^{\alpha} + \int \limits_\Omega \int \limits_0^1 \frac{\lambda|\rho \tau_{s,h} \D u_k|^2}{(1+ |\D u_k + t \tau_{s,h} \D u_k|^2)^{\mu/2}} \dd{t} \dd{x} + \frac{1}{A_k k^2} \int \limits_{\Omega} |\rho \tau_{s,h} \D u_k|^2 \dd{x}. \label{est-blw-1}
\end{align}

By the Cauchy-Schwarz inequality, for $0 \le t \le 1$ and any $a, b \in \mathbb{R}^{N \times n}$, there exists a constant $C > 0$, independent of $t, a$, and $b$, such that
$$(1 + |a + tb|^2)^{1/2} \le C (1 + |a|^2 + |b|^2)^{1/2}.$$

Using this estimate, we proceed to further estimate \eqref{est-blw-1} from below
\begin{multline*}
    \mathit{I'} \ge -C|h|^{\alpha} + \tilde \lambda \int \limits_\Omega \frac{|\rho \tau_{s,h}\D u_k|^2}{(1+ |\D u_k(x)|^2 + |\D u_k(x+he_s)|^2)^{\mu/2}}\dd{x}  + \frac{1}{A_k k^2} \int \limits_{\Omega} |\rho \tau_{s,h} \D u_k|^2 \dd{x} =: \mathit{I''}.
\end{multline*}
At this point we introduce the notation
\begin{align}\label{Notation}
M_{h,s}(x) := 1 + |\D u_k(x+he_s)|^2 + |\D u_k(x)|^2.
\end{align}
For $M \in \N$, we will work with the auxiliary function $V_\kappa: \mathbb{R}^{M} \to \mathbb{R}^{M}$ given by 
$$V_\kappa(\xi) := (1+ |\xi|^2)^{\frac{1-\kappa}{2}}\xi, \quad \xi \in \mathbb{R}^M.$$
We remark that the comparison estimate
\begin{align}\label{eq:comparison}
|\tau_{s,h}V_\kappa(v(x))| \sim (1 + |v(x+he_s)|^2 + |v(x)|^2)^{\frac{1-\kappa}{2}} |\tau_{s,h} v(x)|
\end{align}
holds for any measurable map $v$ and for $1 < \kappa < 2$, where $e_s \in \mathbb{R}^n$ is a unit vector (see \cite[Lemma 2.4]{GmeinederKristensen19BD}). We note that the parameter $\alpha$ in Lemma 2.4 of \cite{GmeinederKristensen19BD} corresponds to the parameter  $\kappa$ in our current setup. This allows us to derive
\begin{multline*}
    \mathit{I} \ge \mathit{I''} \ge -C|h|^{\alpha} + c(\tilde \lambda) \int \limits_\Omega \frac{|\rho \tau_{h,s} V_\kappa(\D u_k(x))|^2}{M_{h,s}(x)^{\frac{2(1-\kappa)+\mu}{2}}}\dd{x} + \frac{1}{A_k k^2} \int \limits_{\Omega} |\rho \tau_{s,h} \D u_k|^2 \dd{x}   =: \mathit{I}_1'' + \mathit{I}_2'' + \mathit{I}_3''.
\end{multline*}

The term $\mathit{I}_1'' = -C|h|^{\alpha}$ is moved from the lower bound for $\mathit{I}$ onto the right-hand side. If $\delta < 1$, we can absorb $\mathit{IV}_1^{(\delta)}$ into $\mathit{I}_3''$, 
leading to the inequality
\begin{multline*}
    c(\tilde \lambda) \int_\Omega \frac{|\rho \tau_{h,s} V_\kappa(\D u_k(x))|^2}{M_{h,s}(x)^{\frac{2(1-\kappa) + \mu}{2}}} \, \mathrm{d}x + \frac{1 - \delta}{A_k k^2} \int_{\Omega} |\rho \tau_{s,h} \D u_k|^2 \, \mathrm{d}x \\
    \leq C\left( |h|^{1+ \alpha}+|h|+  \frac{|h|^2}{k} +|h|^\alpha \right) + C(\delta, p, \Omega) |h|^2.
\end{multline*}
Since the second term on the left hand side is finite and we may assume without loss of generality that $|h|<1$, we can divide by $|h|^{ \alpha}$ to obtain
\[
\int_\Omega \frac{|\rho \tau_{h,s} V_\kappa(\D u_k(x))|^2}{|h|^{ \alpha} M_{h,s}(x)^{\frac{2(1-\kappa) + \mu}{2}}} \, \mathrm{d}x \le C < \infty,
\]
where $C$ does not depend on $k$. Therefore, we have arrived at
\begin{align}\label{3.31.1}
\sup_{k \in \N} \int \limits_\Omega \left \lvert \frac{\rho \tau_{s,h}V_\kappa(\D u_k(x))}{|h|^{\frac{\alpha}{2}}}\right \rvert^2 \frac{1}{\left(1+ |\D u_k(x+he_s)|^2 + |\D u_k(x)|^2\right)^{\frac{2(1-\kappa) + \mu}{2}}} \dd{x} <\infty.
\end{align}
Now, to establish higher integrability of the gradients $\D u_k$, 
 we proceed by utilizing \eqref{eq:comparison} and obtain
\begin{align}\label{eq:YoungNikolskii}
    \int \limits_\Omega \rho \frac{|\tau_{s,h} V_\kappa(\D u_k(x))|}{|h|^{\frac{\alpha}{2}}} \dd{x}\notag 
    & =\int \limits_\Omega \rho \frac{|\tau_{s,h} V_\kappa(\D u_k(x))|}{|h|^{\frac{\alpha}{2}}} M_{h,s}^{-\frac{\mu + 2(1-\kappa)}{4}} M_{h,s}^{\frac{\mu + 2(1-\kappa)}{4}} \dd{x}\notag \\
    \le &\frac 12 \int \limits_\Omega \rho^2 \frac{|\tau_{s,h}V_\kappa(\D u_k)|^2}{|h|^{\alpha}}\frac{1}{M_{h,s}(x)^{\frac{\mu + 2(1-\kappa)}{2}}} \dd{x} + \frac 12 \int \limits_\Omega M_{h,s}(x)^{\frac{\mu + 2(1-\kappa)}{2}} \dd{x}.
\end{align}
The first term in the last line is bounded by \eqref{3.31.1}, and the second term is uniformly bounded if $\mu + 2(1 - \kappa) \le 1$ because the sequence $(u_k)_k$ is uniformly bounded in $\sobo^{1,1}(\Omega; \R^N)$. This condition translates to $\frac{\mu + 1}{2} \le \kappa$. Note that all this is valid under the constraint $1 < \kappa < 2$, where the lower bound is automatically satisfied due to $\mu > 1$, and the upper bound requires $\mu < 3$.
Summarizing, we have:
\[
\sup_{k \in \N} \int \limits_\Omega \rho \frac{|\tau_{s,h} V_\kappa(\D u_k(x))|}{|h|^{\frac{\alpha}{2}}} \, \mathrm{d}x < \infty,
\]
which, due to the arbitrariness of $\rho$ and the direction $s$, implies that the sequence $(V_\kappa(\D u_k))_k$ is uniformly bounded in the Besov space $\displaystyle \B^{\alpha/2}_{1,\infty}(K; \R^N)$ for any relatively compact Lipschitz subset $K \Subset \Omega$.
According to function space theory (see, for example,  \cite[Theorem 7.34]{adams_sobolev_2003} in conjunction with embeddings of weak $\lebe^p$-spaces), this implies that $(V_\kappa(\D u_k))_k$ is uniformly bounded in $\lebe^{\frac{2n}{2n - \alpha} - \delta'}(K; \R^N)$ for any $0 < \delta' < \frac{2n}{2n - \alpha}$.
Furthermore, by \cite[Lemma 2.4]{GmeinederKristensen19BD} with $p$ being $\frac{2n}{2n- \alpha} - \delta'$ in our case, we conclude that for any relatively compact Lipschitz set $K \Subset \Omega$, there exists a constant $C(\alpha, \delta', \kappa,K)$ such that
\[
\sup_{k \in \N} \int \limits_K |\D u_k|^{(2 - \kappa) \left( \frac{2n}{2n - \alpha} - \delta' \right)} \, \mathrm{d}x = C(\alpha, \delta', \kappa,K) < \infty.
\]
We now choose $\kappa$ and $\delta'$ in a suitable way. Notice that the condition
\[
(2 - \kappa) \frac{2n}{2n - \alpha} > 1 \Leftrightarrow \kappa < 1 + \frac{\alpha}{2n}
\]
should be satisfied to achieve higher integrability. Considering the lower bound for $\kappa$ derived earlier, we arrive at the condition
\[
\frac{\mu + 1}{2} < 1 + \frac{ \alpha}{2n} \Leftrightarrow \mu < 1 + \frac{\alpha}{n}.
\]
 Therefore, for $\mu < 1 + \frac{\alpha}{n}$, we can choose a $\kappa$ that satisfies both bounds. Having chosen such a $\kappa$, we then select $\delta' > 0$ small enough so that
\[
(2 - \kappa) \left( \frac{2n}{2n - \alpha} - \delta' \right)>1.
\]
This way we have deduced the uniform bound
\begin{align}\label{Lpuniform}
\sup_{k \in \N} \int \limits_K |\D u_k|^p \, \mathrm{d}x < \infty
\end{align}
on any relatively compact Lipschitz set $K \Subset \Omega$ for the exponent 
\[
p := (2 - \kappa) \left( \frac{2n}{2n - \alpha} - \delta' \right) > 1.
\]
Since we can let $\delta' \to 0$ and optimize $\kappa$ within the given bounds, it follows that the sequence $(\D u_k)_k$ is bounded in $\lebe^q(K;\R^{N \times n})$ for every $q < \frac{(3 - \mu)n}{2n - \alpha}$.

By the weak$^*$ convergence of $\D u_k$ to $\D u$, we may now apply the lower semicontinuity part of Reshetnyak's theorem (Theorem \ref{thm:Reshetnyak}) for the convex function $g(z) = |z|^q$ as explained in \eqref{eq:ReshetnyakApplication} to obtain:
$$\int \limits_K |\nabla u|^q \dd{x} + \int \limits_K g^\infty\left(\frac{\dd{\D^su}}{\dd{|\D^s u|}}\right) \dd{|\D^su|} \le \liminf_{k \to \infty} \int \limits_K |\D u_k|^q \dd{x} < \infty.$$
Since $q>1$ we have $g^\infty(z) \equiv + \infty$, and this implies that $\D^s u$ vanishes in $K$. By arbitrariness of $K \Subset \Omega$ we even have $\D^su=0$ in $\Omega$ and $\nabla u = \D u$. Hence 
$$\int \limits_K |\D u|^q \dd{x} <\infty,$$
which implies $\D u \in \lebe^q_{\mathrm{loc}}(\Omega; \R^{N \times n})$ for all $q < \frac{(3-\mu)n}{2n-\alpha}$. By Poincaré's inequality we infer $u \in \sobo^{1,q}_{\mathrm{loc}}(\Omega; \R^N)$ for the same range of $q$. The proof is complete.  \qed\\

Alternatively, in the very last step we could argue that for $n \ge 2$ we have $\bv(\Omega; \R^N) \hookrightarrow \lebe^{\frac{n}{n-1}}(\Omega; \R^N)$, and note that $\frac{(3-\mu)n}{2n-\alpha} < \frac{n}{n-1}$ for all admissible $\mu$.

\section{Lack of regularity in one dimension for $\mu>1+\alpha$}
\label{sharpness}

In this section, we provide a one-dimensional example, demonstrating that the range for the ellipticity parameter $\mu$ in Theorem \ref{maintheorem} is essentially sharp, reflecting the lack of regularity for higher ellipticity parameters in this setting. To that end, we show that for $n=N=1$, whenever $ \mu > \alpha + 1$, there is a generalized minimizer to a linear growth functional of the form considered in Theorem \ref{maintheorem}, which is of class $\bv\backslash \sobo^{1,1}_{\mathrm{loc}}(I, \R)$. Our construction is strongly inspired by previous and by now classical examples involving slightly different functionals, see \cite[p. 132]{Kaiser75}, \cite[Example 3.1]{GiaquintaModicaSoucek79-2} and \cite[Example 5.2]{HakkarainenKinnunenLahti}.

For this, we consider weighted integrands of the form
\begin{equation}\label{1Dproblem}
\mathscr F[u; I]: = \int \limits_a^b f(u'(x)) w(x) \dd x
\end{equation}
on the interval $I= (a,b)$ for $f \in \hold^2(\R)$ strictly convex and of linear growth, and for $w \in \hold([a,b])$ bounded below by some $m>0$. We seek to minimize the functional over Dirichlet classes
$$\mathscr D:= \{u \in \sobo^{1,1}(a,b): u(a)= y_1, u(b)=y_2\},$$
where due to continuity of Sobolev functions in one dimension the Dirichlet boundary condition is to be understood in the sense of a pointwise evaluation at the boundary points. 
\begin{beisp}
\label{counterexample}
Let $\mu \in (1,2), \alpha \in (0,1)$ with $\mu > \alpha +1$. Consider the integrand $F: [-1,1] \times \mathbb{R} \to \mathbb{R}$ given by $F(x,z) = w(x) f(z)$, where \[
w(x) = 1 + |x|^\alpha 
\]
and \[
f(z) = \begin{cases}
    \frac{\mu-1}{2\mu} z^2, & |z| \leq 1 \\
    |z| - \frac{1}{\mu(2-\mu)}|z|^{2 - \mu} + \frac{\mu-1}{2(2-\mu)}, & |z| > 1.
\end{cases}
\]
Then there exists $\tilde{M} > 0$ such that for any $M > \tilde{M}$, the generalized minimizer for the functional \eqref{1Dproblem} with boundary conditions $u(-1) = 0$ and $u(1) = M$ is not in $\sobo^{1,1}_{\mathrm{loc}}(-1,1)$ and exhibits a jump discontinuity at $x=0$.
\end{beisp}

This shows that the range $\mu < 1+\alpha$ of Theorem \ref{maintheorem} is essentially sharp, since the integrand satisfies all hypotheses of the Theorem except the range for $\mu$.\\

To prove the example we give a more abstract result on how to construct minimizers with jumps in one dimension.

\begin{prop}
\label{prop:counter}
  Let $f$ additionally satisfy the following conditions:
  \begin{enumerate}
    \item $\min f(\R) = 0 = f(0)$.
    \item $\lim_{t\to + \infty} f(t)/t = 1$ and $\lim_{t \to -\infty} f(t)/t = -1$, or equivalently $f'(\R) = (-1,1)$.
    \item $y_1=0, y_2=M>0$.
\end{enumerate}
    Assume in addition that $w$ attains its minimum $m$ at exactly one point $c \in (a,b)$. Defining $g:= (f')^{-1}: (-1,1) \to \R$, let
    $$M_0 := \int \limits_a^b g\left(\frac{m}{w(t)}\right) \dd t,$$
    and assume that $M_0$ is finite. If $M>M_0$, then any generalized minimizer has a jump at $c$ of size at least $M-M_0$.
\end{prop}

\begin{proof}
    {If $M>M_0$, let $u \in \bv(a,b)$ be a generalized minimizer of \eqref{1Dproblem}, which exists by the standard direct method. Specifically, $u$ is a minimizer of the relaxed functional, where $\tilde u$  denotes the prescribed boundary values.} 
    \begin{align*}
        \overline{\mathscr F}[u]:= \int \limits_a^b f(\nabla u(x)) w(x) \dd x + \int \limits_a^b f^\infty \left( \frac{\dd \D^su}{ \dd \left|\D^su\right|} \right) w \ \dd \left|\D^su\right| + \int_{\partial (a,b)} f^\infty(u-\tilde u) w \ \dd \mathscr H^0.
    \end{align*}

    {The recession function in this context is given by:}
    $$f^\infty(z) = \limsup_{t \to \infty} \frac{f(tz)}{t} = \limsup_{t \to \infty} z \frac{f(tz)}{tz} = |z|$$
    {where the final equality follows from the second assumption on $f$ in the Proposition \ref{prop:counter}.} 
    Therefore we have
    $$\overline{\mathscr F}[u]= \int \limits_a^b f(\nabla u(x)) w(x) \dd x + \int \limits_a^b w \dd |\D^su| + w(a) |u(a)| + w(b)|u(b)-M|.$$
    
    We now show that $u$ attains its boundary values. For this purpose, consider the competitor
    $$v= \begin{cases}
        u-u(a), &\text{on } [a,c)\\
        u - u(b) + M, &\text{on } [c,b],\\
    \end{cases}$$
    for which we obtain by minimality of $u$
    \begin{align*}
        \overline{ \mathscr F}[u] \le \overline{\mathscr F}[v] = \int \limits_a^b f(\nabla u(x)) w(x) \dd x + \int \limits_{(a,b)\backslash \{c\}} w \dd |\D^su| + w(c) |u(c+)-u(b)+M -u(c-)+u(a)|\\
        \le \int \limits_a^b f(\nabla u(x)) w(x) \dd x + \int \limits_a^b w \dd |\D^su| + w(a)|u(a)| + w(b)|u(b)-M| =\overline{\mathscr F}[u],
    \end{align*}
    where we used the triangle inequality in the last line as well as $0<w(c)< w(a)$ and $0<w(c)< w(b)$. The second inequality therefore has to be an equality which can only hold if
    $$(w(a)-w(c))|u(a)| = (w(b) -w(c))|u(b)-J|=0,$$
    implying $u(b)=M$ and $u(a)=0$ by $w^{-1}(\{m\})=c$.\\

    Next we show that $\D^su$ must concentrate at $c$. 
    {Suppose not; then let}  $v \in \bv(a,b)$ be such that $\D v = \nabla u \mathscr L^1 + \left(\int \limits_a^b \dd \D^su\right) \delta_c$ with $v(a) = 0$. Such a function exists by integrating $\nabla u$ and adding a jump of appropriate size at $c$. We then have by construction $v(b)=u(b)=M$ and hence $v$ is admissible. We calculate
    $$\overline{\mathscr F}[v] = \int \limits_a^b f(\nabla u) w \dd x +w(c) \left | \int \limits_a^b \dd \D^su\right | \le \int \limits_a^b f(\nabla u)w \dd x+ \int \limits_a^b w \dd |\D^su| = \overline{\mathscr F}[u],$$
    where we used the definition of $c$ as the unique minimum point for $w$ and the triangle inequality for the total variation measure. By minimality of $u$ we have equality which is only attained (again by definition of $c$) if $\mathrm{supp}(\D^su) \subset w^{-1}(m) = \{c\}$. This means that it suffices to minimize over maps $v \in \bv(a,b)$ such that
    $$v= v_0 + k\chi_{[c,b]}$$
    with $k \in \R$, $v_0 \in \sobo^{1,1}$ with $v_0(a)=0$. This is due to the fact that the Cantor part of the derivative cannot concentrate in one single point. Therefore, only a jump can occur. In this case, the jump height $k$ is determined by $M= \int \limits_a^b v_0'(t) \dd t + k$, so that
    $$\overline{\mathscr F}[v] = \int \limits_a^b f(v_0'(t))w(t) \dd t+ w(c)|k|$$
We claim that $k\ge M-M_0>0$. To that end, depending on whether $k>0$ or $k \le 0$, we have
$$\overline{\mathscr F}[v] = \underbrace{\int \limits_a^b f(v_0'(t))w(t) \mp m v_0'(t) \dd t}_{\mathscr J[v_0]} \pm mM.$$
It therefore suffices to find an extremal of $\mathscr J$ in $\{v_0 \in \sobo^{1,1}(a,b): v_0(a)=0\}$. The Euler-Lagrange equation is $\left(f'(v_0'(t))w(t) \mp m\right)'=0$, hence
$$f'(v_0'(t)) w(t) \mp m=C$$
for some constant $C$ with $C \in (-2m,0)$ if $k>0$ and $C \in (0,2m)$ if $k \le 0$.
This range for $C$ can be checked by plugging in $t=c$ and using $f'(\R) = (-1,1)$. We can solve this to obtain
$$v_0(x) = \int \limits_a^x g\left(\frac{C\pm m}{w(t)}\right)\dd t.$$
Since for fixed $x$, the right hand side is increasing in $C$, we obtain
$$v_0(b)\le \int \limits_a^bg\left(\frac{m}{w(t)}\right)\dd t = M_0.$$
This implies that $k=M- \int \limits_a^b v_0'(t) \dd t \ge M-M_0 >0$.\\
We have hence shown that a jump of size $k\ge M-M_0>0$ at $c= w^{-1}({m})$ occurs, if the right hand side boundary value $M$ is larger than $M_0$.    

\end{proof}

\begin{figure}[t]
    \centering
    \includegraphics[scale=0.8]{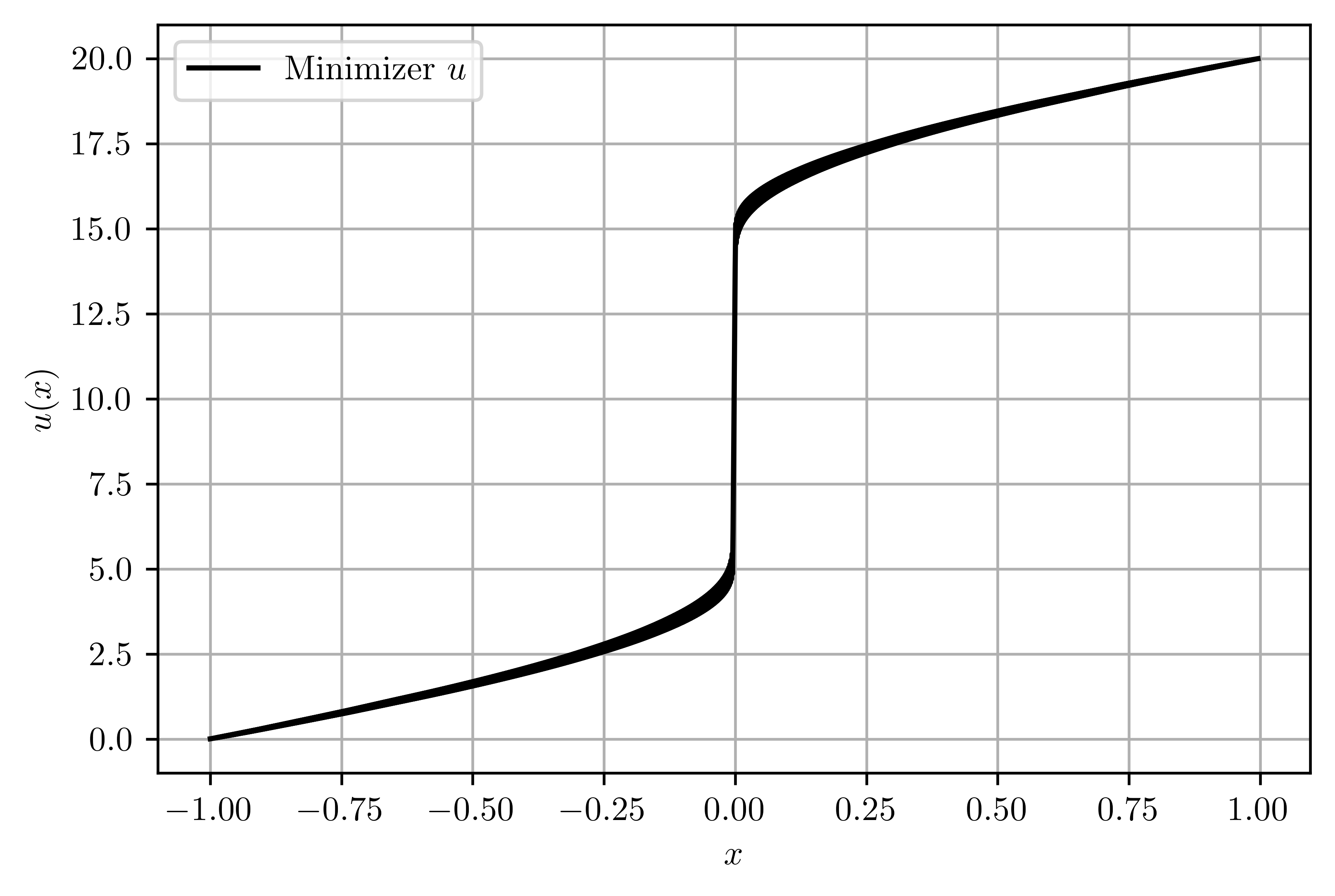}
    \caption{Numerical simulation of Example \ref{counterexample} with parameters $\mu = 1.4$, $\alpha = 0.25$ and $M=20$.}
    \label{fig:enter-label}
\end{figure}

\begin{proof}[Proof of Example \ref{counterexample}]
We have \[
f'(z) = \begin{cases}
    \frac{\mu-1}{\mu} z & |z|\leq1 \\
    \frac{z}{|z|} - \frac{z}{\mu|z|^\mu} & |z| > 1
\end{cases}
\]
and \[
g(z) = (f')^{-1}(z) = \begin{cases}
    ´\frac{\mu}{\mu-1}z & |z| \leq \frac{\mu-1}{\mu} \\
    -(\mu(z+1))^{\frac{1}{1-\mu}} & -1 < z < \frac{1-\mu}{\mu} \\
    (\mu(1-z))^\frac{1}{1-\mu} & 1 > z > \frac{\mu-1}{\mu}.
\end{cases}
\]
Since we have $m:=\min_{[-1,1]} w = 1$ and $\frac{\mu-1}{\mu}<\frac{1}{2}$, we obtain taking advantage of $\frac 12 < \frac{1}{1+|x|^\alpha} \le 1$ on $(-1,1)$, that \begin{align*}
M_0 & = \int \limits_{-1}^1 \left(\mu\left(1-\frac{1}{1+|z|^\alpha}\right)\right)^\frac{1}{1-\mu} \dd{z} \\
& = 2 \mu^{\frac{1}{1-\mu}} \int \limits_0^1 \left(\frac{t^\alpha}{1+t^\alpha}\right)^\frac{1}{1-\mu} \dd{t}.
\end{align*}
This integral is finite if and only if \[
\frac{\alpha}{1-\mu} > -1,
\]
which is equivalent to $\mu>\alpha +1 $.
We also note that \[
f''(z) = \begin{cases}
    \frac{\mu-1}{\mu} & |z| \leq 1 \\
    \frac{\mu-1}{\mu}|z|^{-\mu} & |z| > 1,
\end{cases}
\]
so $f \in \mathrm{C}^2(\mathbb{R})$ is strictly convex and $\mu$-elliptic. 

{Furthermore, $f'$ is bounded, implying that $f$ is of linear growth. The weight $w$, on the other hand, is continuous and bounded from below by $m = 1$. The function $f$ achieves its global minimum at $z=0$, and satisfies $\lim \limits_{t \to \infty} f'(t) = 1$ and $\lim \limits_{t \to -\infty} f'(t) = -1$. Consequently, Proposition \ref{prop:counter} is applicable. The preceding calculations demonstrate that for $\mu >\alpha +1$, there exists a $\bv$-minimizer that does not belong to $\sobo^{1,1}_{\mathrm{loc}}$, since the necessary integral $M_0$ is finite, thereby confirming the existence of $\tilde M$ as claimed.}

\end{proof}
Note that all assumptions of Theorem \ref{maintheorem} are fulfilled by $F$ from the example. The linear growth and $\mu$-ellipticity are immediate from the calculations, while the boundedness of $f'(z)$ implies the Hölder continuity condition \ref{cond:3}, since $w(x)$ is obviously $\alpha$-Hölder continuous.

\section{Dimension bound - Proof of Theorem \ref{maintheorem2}}\label{sec:proofDimension}
The estimates achieved in the proof of Theorem \ref{maintheorem} lead to higher differentiability, which can be seen by the minimizers $u$ belonging to fractional Sobolev spaces. By the measure density lemma we can transform this fractional differentiability into a dimension bound for the set of non-Lebesgue points of the derivative $\D u$ of a generalized minimizer $u$. In view of the partial regularity results (see the representation \eqref{eq:singularset}) this is equivalent to a dimension bound for the singular set.
\begin{proof}[Proof of Theorem \ref{maintheorem2}]
We go back to the uniform estimate \eqref{3.31.1} established in the proof of Theorem \ref{maintheorem} for the larger regime $1 < \mu < 1+ \frac{\alpha}{n}$. Similarly as in \eqref{eq:YoungNikolskii} above, we use Young's inequality to estimate on a relatively compact Lipschitz subset $K \Subset \Omega$: 
\begin{align}\label{eq:YoungNikolskii2}
    \int \limits_K \frac{|\tau_{s,h}V_\kappa(\D u_k(x))|}{|h|^{\frac{\alpha}{2}}M_{h,s}(x)^{\frac{1-\kappa}{2}}} \dd{x} 
    & =\int \limits_K  \frac{|\tau_{s,h}V_\kappa(\D u_k(x))|}{|h|^{\frac{\alpha}{2}}M_{h,s}(x)^{\frac{1-\kappa}{2}}} M_{h,s}^{-\mu/4}(x) M_{h,s}^{\mu/4}(x)\dd{x}\notag \\
    & \le \frac 12 \int \limits_K \frac{|\tau_{s,h}V_\kappa(\D u_k(x))|^2}{|h|^{ \alpha}}\frac{1}{M_{h,s}(x)^{\frac{\mu + 2(1-\kappa)}{2}}} \dd{x} + \frac 12 \int \limits_K M_{h,s}(x)^{\mu/2} \dd{x}.
\end{align}
Here, the first term is uniformly bounded by \eqref{3.31.1}, and the second term is uniformly bounded if and only if $\mu < \frac{(3-\mu) n}{2n-\alpha}$, as seen in Theorem \ref{maintheorem}. This is equivalent to $\mu < \frac{3n}{3n-\alpha}.$

For this range of $\mu$, according to \eqref{eq:comparison}, inequality \eqref{eq:YoungNikolskii2} is equivalent to
$$\sup_{k \in \N} \int \limits_K  \frac{|\tau_{s,h}\D u_k(x)|}{|h|^{\frac{\alpha}{2}}}<\infty$$
which implies that the sequence $(\D u_k)_k$ is uniformly bounded in the Nikolskiǐ space $\mathcal{N}^{\frac{\alpha}{2}, 1}(K; \R^N)$. By definition, this space coincides with the Besov space $\B_{1, \infty}^{\alpha/2}(K;\R^N)$, which, by standard interpolation theory, is continuously embedded into $\B_{1, 1}^\theta(K;\R^N)$ for all $\theta < \frac{\alpha}{2}$ (see, for example, \cite[2.3.2, Prop. 2]{Triebel}). $\B_{1, 1}^\theta(K;\R^N)$, in turn, is isomorphic to $\sobo^{\theta,1}(K;\R^N)$
, implying that $(\D u_k)$ is uniformly bounded in the fractional Sobolev spaces $\sobo^{\theta,1}(K;\R^N)$ for all $\theta < \frac{\alpha}{2}$.
\\

By the compact embedding $\sobo^{\theta,1}(K; \R^N) \hookrightarrow\hookrightarrow\lebe^1(K; \R^N)$ we infer the existence of a subsequence of $(\D u_k)_k$ converging strongly in $\lebe^1(K; \R^N)$ and pointwise a.e. to some limit $v$. By the weak$^*$-convergence of $(\D u_k)_k$ to $\D u$ we obtain $\D u =v$ a.e. in $K$, and hence by Fatou's lemma
 \begin{align*}
    \int \limits_K \frac{|\tau_{s,h}\D u|}{|h|^{\frac{\alpha}{2}}}\dd{x} \le \liminf_{k \to \infty} \int \limits_K \frac{|\tau_{s,h}\D u_k|}{|h|^{\frac{\alpha}{2}}}\dd{x} \le \sup_{k \in \N} \sup_{0< |h| <\min\{1, \tilde \delta\}}\int \limits_K \frac{|\tau_{s,h}\D u_k|}{|h|^{\frac{\alpha}{2}}}\dd{x} <\infty,
 \end{align*}
 where $\tilde{\delta}$ is chosen such that $x+\tilde{\delta}e_s \in \Omega$ for all $x \in K$, $s \in \{1, \dots, n\}$.
 Taking the supremum over all $h$ in a suitable range and all $s \in \{1, \dots, n\}$ on the left hand side yields $\D u \in \sobo^{\theta,1}(K;\R^N)$ for all $\theta < \frac{\alpha}{2}$ following the same reasoning as for $\D u_k$.  
 
Since we can exhaust $\Omega$ by countably many relatively compact subsets $K$, we therefore have by Proposition \ref{thm:fracsobo-hausdorff}
$$\dim_{\mathscr H}(\mathrm{Sing}(u)) \leq n- \frac{\alpha}{2} .$$
This concludes the proof.
\end{proof}

As noted before, the Sobolev regularity result holds in the broader range $1 <\mu < 1+ \frac{ \alpha}{n}$. However, to achieve the necessary smoothness required for dimension reduction, the ellipticity parameter $\mu$ must be selected based on the precise value of the maximal integrability. At this stage in the proof of Theorem \ref{maintheorem2}, the more restrictive condition $1< \mu <\frac{3n}{3n-\alpha}$ becomes essential.\\

As alluded to in the introduction, a better dimension bound is available in the case of autonomous $\mu$-elliptic linear growth integrands. Since it is hard to find a proof in the literature, we provide a sketch here. The underlying ideas are, in principle, not novel and can be partially found in the literature; see, for instance, \textsc{Beck \& Schmidt} \cite{BeckSchmidt13} and \textsc{Gmeineder} \cite{Gmeineder16,Gmeineder20BD}.
\begin{satz}[The autonomous case]\label{thm:autonomous}
    Let $F: \R^{N \times n} \to \R$ be a linear growth integrand which is $\mu$-elliptic for $1< \mu \le \frac{n}{n-1}$ if $n\ge 3$ or for $1<\mu <2$ if $n=2$. Then every generalized minimizer $u \in \bv(\Omega; \R^N)$ to the variational principle
    $$\mathscr F[u; \Omega] = \int \limits_\Omega F(\D u) \dd{x}, \quad u \in \sobo^{1,1}_{u_0}(\Omega; \R^N)$$
    satisfies
    $$\dim_{\mathscr H}(\operatorname{Sing} u) \le n-1.$$
\end{satz}
Note that, as a consequence of Proposition \ref{prop:PartialRegularity}, every generalized minimizer is partially $\hold^{1,\alpha}$-regular.
\begin{proof}[Sketch of proof]
    As a first step, we need to establish that for the vanishing viscosity approximating sequence $(u_k)_k$, 
    {constructed with the help of Ekeland's variational principle, we have the following uniform bound}
    \begin{equation}\label{eq:autonomous1}
    \int \limits_{U} \frac{|\D^2 u_k|^2}{\left(1+|\D u_k|^2\right)^{\mu/2}} \dd{x} \le C(U, \mu)<\infty 
    \end{equation}
    for every Lipschitz subset $U \Subset \Omega$. This is a consequence of the approximate Euler-Lagrange inequality satisfied by the sequence $(u_k)_k$. The reasoning follows similarly to the proof of Theorem \ref{maintheorem}. For the autonomous case, the argument is essentially contained in \cite[Lemmas 5.1, 5.2]{BeckSchmidt13}, if one exchanges the ellipticity exponent $3$ used by the authors in \cite{BeckSchmidt13} by $\mu$ in our case and the integrability $p$ of $u_k$ by $2$. Having established this, we proceed as follows:    
    By Young's inequality, we have for any $U \Subset \Omega$
    \begin{equation}\label{eq:autonomous2}
    \int \limits_U |\D^2 u_k| \dd{x} \le C \int \limits_U \frac{|\D^2 u_k|^2}{\left(1+|\D u_k|^2\right)^{\mu/2}} \dd{x} + \int \limits_U (1+|\D u_k|^2)^{\mu/2} \dd{x}.
    \end{equation}
   In order to find out when the right-hand side is uniformly bounded in $k$, we employ the auxiliary function 
   $$\tilde V_\mu(\xi):= (1+|\xi|^2)^{\frac{2-\mu}{4}}.$$
A straightforward application of the chain rule yields the estimate
   $$\abs{\partial_{k} \tilde V_\mu(\D w)}^2\le C(\mu)(1+|\D w|^2)^{-\mu /2} \abs{\partial_k \D w}^2, \quad k = 1, \dots, n, $$
   from which the uniform bound recorded in \eqref{eq:autonomous1} implies the uniform boundedness of $\tilde V_\mu(\D u_k)$ in $\sobo^{1,2}_{\mathrm{loc}}(\Omega; \R^N)$. By the Sobolev embedding theorem, this further implies the uniform boundedness of $(\tilde V_\mu(\D u_k))$ in $\lebe^{\frac{2n}{n-2}}_{\mathrm{loc}}$ for $n \ge 3$ or in $\operatorname{BMO}_{\mathrm{loc}}$ for $n =2$. By the definition of the auxiliary function $\tilde V_\mu$, this leads to the bound:
   $$\int \limits_U |\D u_k|^{\frac{(2-\mu)n}{n-2}} \dd{x} \le C(U) < \infty$$
   for $n \ge 3$, where the exponent satisfies $\frac{(2-\mu)n}{n-2}>1$ due to the condition $\mu \le \frac{n}{n-1} < 1+ \frac 2n$.\\

   Referring back to \eqref{eq:autonomous2}, we {deduce} for $n \ge 3$ that the right-hand side is uniformly bounded if and only if
   $$\mu \le \frac{(2-\mu)n}{n-2} \Leftrightarrow \mu \le \frac{n}{n-1}.$$
   For $n = 2$, we utilize the embedding $\operatorname{BMO}_{\mathrm{loc}} \hookrightarrow \lebe^q_{\mathrm{loc}}$ for every $1 \le q < \infty$. {Consequently,} $\tilde V_\mu(\D u_k)$ is uniformly bounded in $\lebe^q_{\mathrm{loc}}$ for every $1 \le q <\infty$. The right-hand side of \eqref{eq:autonomous2} is therefore uniformly bounded if and only if
   $$\mu \le \frac{2-\mu}{2}q \Leftrightarrow \mu \le \frac{2q}{2+q}.$$
   By letting $q \nearrow \infty$ we obtain the range $1 < \mu <2$ for $n=2$. Independently of the dimension, we conclude as follows: From the uniform bound in \eqref{eq:autonomous2}, {it follows that} $(\D u_k)_k$ is uniformly bounded in $\sobo^{1,1}_{\mathrm{loc}}(\Omega; \R^N)$, and hence converges weakly$^*$ to some $v \in \bv_{\mathrm{loc}}(\Omega; \R^N)$. By the convergence properties of $(u_k)_k$ we have $v= \D u$ a.e. in $\Omega$, and $\D u \in \bv_{\mathrm{loc}}(\Omega; \R^N)$. {Subsequently, the measure density lemma, in conjunction with the Poincaré inequality for}  $\bv$,  namely
   $$\fint \limits_{\ball(x_0, r)} |\D u - (\D u)_{\ball(x_0, r)}| \dd{x} \le \frac{|\D^2 u|(\ball(x_0,r))}{\mathscr{L}^n(\ball(x_0,r))^{\frac{n-1}{n}}}$$
   {can be applied} in the following way: if a point belongs to the singular set $\operatorname{Sing} u$, taking the limit superior for $r \searrow 0$ in the {above} inequality, the left hand side stays positive. Then the measure density lemma (Lemma \ref{lem:measuredensity}) applied to the set function $O \mapsto |\D^2 u|(O)$ implies
   $$\dim_{\mathscr H}(\operatorname{Sing} u) \le n-1$$
   upon noting that $\mathscr L^n(\ball(x_0,r))^{\frac{n-1}{n}} \sim r^{n-1}.$
\end{proof}

\section{Possible extensions}\label{sec:extensions}
Let us briefly comment on possible extensions or improvements of Theorems \ref{maintheorem} and \ref{maintheorem2}. First, as mentioned in the introduction, an a priori $\lebe^\infty$-assumption usually allows to obtain Sobolev regularity results in a larger range of $\mu$. More specifically, it would be interesting to show Sobolev regularity for all bounded minimizers to $\mu$-elliptic non-autonomous linear growth functionals for $\mu \le 3$. This would match the ellipticity bounds for the autonomous case and it is the best range of $\mu$ one can hope fore. This is due to a counterexample by \textsc{Bildhauer} \cite[Chapter 4.4]{Bildhauer03book}, based on ideas by \textsc{Giaquinta, Modica \& Souček} \cite{GiaquintaModicaSoucek79-1,GiaquintaModicaSoucek79-2}, showing that for $\mu>3$, Sobolev regularity for all bounded generalized minimizers can not be achieved.\\

Concerning the dimension reduction, it would be interesting to investigate the gap in the dimension bound that our results show in comparison to the autonomus case. For $\alpha \nearrow 1$, we are able to arrive at $\left( n- \frac 12 \right)$ as a bound for the dimension of the singular set, but not at $(n-1)$ as in the autonomus case, see Theorem \ref{thm:autonomous}. It is presently not clear to us if this gap is natural, seeing as a gap in the admissible ellipticity ranges between Lipschitz continuous non-autonomous and autonomous integrands does actually occur (see the case of $(p,q)$-growth in Chapter \ref{sec:context} and \cite{EspositoLeonettiMingione04}). This leaves open interesting questions for future investigations - either to look into closing this gap or to find examples of integrands that confirm the existence of this jump.\\

Another possible generalization would be to handle fully non-autonomous integrals of linear growth. By that, we mean integrands $F: \Omega \times \R^N \times \R^{N \times n} \to \R$ with linear growth in the third variable. The corresponding functionals are of the form
$$\mathscr F[u; \Omega]:= \int \limits_\Omega F(x, u, \D u) \dd{x}, \quad u: \Omega \to \R^N.$$
Similar to the case of non-autonomous integrands as treated by us, one has to work with a suitably relaxed functional defined on $\bv(\Omega; \R^N)$. However, there is presently no regularity theory available for these functionals, neither Sobolev nor (partial) Hölder regularity are known. One particular problem in obtaining estimates like the ones derived in this article is the fact that the approximate Euler-Lagrange inequality takes a more complicated form. This is because the Euler-Lagrange equation for a generalized minimizer also contains a derivative with respect to the second variable. In view of this, there is little hope to obtain regularity results for fully non-autonomous linear growth functionals with the methods presented here. Nonetheless, it would be interesting to explore these functionals in terms of regularity.

\section*{Acknowledgements}

The authors thank Franz Gmeineder for suggesting the topic and fruitful discussions on the theme of the article. L. F. gratefully acknowledges financial support by the Hector foundation and the Studienstiftung des deutschen Volkes. The research of B. P. was partially supported by YSF offered by University of Konstanz under the project number: FP 638/23. P. S. acknowledges the financial support through a stipend of the Landesgraduiertenf\"orderungsgesetz Baden-W\"urttemberg. The authors are also grateful to the anonymous referee for a careful reading of the manuscript and for valuable comments.

\bibliographystyle{acm}
\bibliography{main.bib}
\end{document}